\documentclass[11pt]{amsart}
\usepackage{float}
\usepackage[usenames]{color}
\usepackage{graphicx}
\usepackage{amssymb}
\usepackage{amscd}
\usepackage{makecell}
\theoremstyle{plain}
\newtheorem{thm}{Theorem}[section]
\newtheorem{lemma}[thm]{Lemma}

\newtheorem{prop}[thm]{Proposition}

\theoremstyle{definition}

\newtheorem{rmk}[thm]{Remark}
\newtheorem{example}[thm]{Example}

\def\dim{\mathop{\hbox {dim}}\nolimits}

\def\ker{\mathop{\hbox{Ker}}\nolimits}

\newcommand{\fra}{\mathfrak{a}}

\newcommand{\frg}{\mathfrak{g}}
\newcommand{\frh}{\mathfrak{h}}

\newcommand{\frk}{\mathfrak{k}}
\newcommand{\frl}{\mathfrak{l}}

\newcommand{\frp}{\mathfrak{p}}
\newcommand{\frq}{\mathfrak{q}}

\newcommand{\frt}{\mathfrak{t}}
\newcommand{\fru}{\mathfrak{u}}

\newcommand{\bbC}{\mathbb{C}}

\newcommand{\bbR}{\mathbb{R}}

\newcommand{\bbZ}{\mathbb{Z}}

\newcommand{\caC}{\mathcal{C}}

\newcommand{\caL}{\mathcal{L}}

\newcommand{\caR}{\mathcal{R}}

\begin{document}

\title{Dirac series of $E_{7(-5)}$}

\author{Yi-Hao Ding}
\address[Ding]{School of Mathematical Sciences, Soochow University, Suzhou 215006,
P.~R.~China}
\email{435025738@qq.com}

\author{Chao-Ping Dong}
\address[Dong]{School of Mathematical Sciences, Soochow University, Suzhou 215006,
P.~R.~China}
\email{chaopindong@163.com}

\author{Ping-Yuan Li}
\address[Li]{School of Mathematical Sciences, Soochow University, Suzhou 215006,
P.~R.~China}
\email{1994959713@qq.com}

\abstract{Using the sharpened Helgason-Johnson bound, this paper classifies all the irreducible unitary representations with non-zero Dirac cohomology of $E_{7(-5)}$. As an application, we find that the cancellation between the even part and the odd part of the Dirac cohomology continues to happen for certain unitary representations of $E_{7(-5)}$. Assuming the infinitesimal character being integral, we further improve the Helgason-Johnson bound for $E_{7(-5)}$. This should help people to understand (part of) the unitary dual of this group.}

\endabstract



\maketitle


\section{Introduction}

A fundamental problem in the representation theory of Lie groups is to understand the unitary dual $\widehat{G}$ of a real reductive Lie group $G$. Here by $\widehat{G}$ we mean the set of the equivalence classes of all the irreducible unitarizable $(\frg, K)$ modules. For simplicity, we assume that $G$ is a connected semisimple Lie group with finite center. Given an irreducible $(\frg, K)$ module $\pi$, it is highly possible that $\pi$ is non-unitary, that is, it does not possess a positive definite invariant Hermitian form. Therefore, to arrive at the unitary dual of $G$ from its admissible dual as classified by Langlands in 1970s, one needs an effective way to rule out the great many non-unitary representations.

Parthasarathy's Dirac operator inequality \cite{Pa2} is a valuable tool on this aspect. Indeed, let $\theta$ be a Cartan involution of $G$ and assume that $K:=G^{\theta}$ is a maximal compact subgroup of $G$. Let
\begin{equation}\label{Cartan-decomposition}
\frg_0=\frk_0 \oplus \frp_0
\end{equation}
be the corresponding Cartan decomposition on the Lie algebra level. As usual, we will drop the subscripts to stand for the complexified Lie algebras. Fix a non-degenerate invariant symmetric bilinear form $B(\cdot, \cdot)$ on $\frg$, which is a scalar multiple of the Killing form on $[\frg, \frg]$ so that each simple coroot has length $2$ when $\frg$ is simply-laced. Let $T_{f}$ be a maximal torus of $K$, and put $A_f:=\exp(\fra_{f, 0})$, where $\fra_{f, 0}$ is the centralizer of $\frt_{f, 0}$ in $\frp_0$. Then up to conjugation, $H_f:=T_f A_f$ is the unique $\theta$-stable maximally compact Cartan subgroup of $G$. Namely, $H_f$ is the fundamental Cartan subgroup of $G$. Let $W(\frg, \frh_f)$ (resp., $W(\frg, \frt_f)$, $W(\frk, \frt_f)$) be the Weyl group of $\Delta(\frg, \frh_f)$ (resp., $\Delta(\frg, \frt_f)$, $\Delta(\frk, \frt_f)$). We fix a positive root system $\Delta^+(\frk, \frt_f)$ once for all. Choose a positive restricted root system $\Delta^+(\frg, \frt_f)$ containing the $\Delta^+(\frk, \frt_f)$. Denote by $\rho$ (resp., $\rho_c$) the half sum of roots in $\Delta^+(\frg, \frt_f)$ (resp., $\Delta^+(\frk, \frt_f)$).

Choose an orthonormal basis $Z_1, \dots, Z_m$ of $\frp_0$. Then the Dirac operator introduced by Parthasarathy \cite{Pa} is
\begin{equation}\label{Dirac-operator}
D:=\sum_{i=1}^{m} Z_i\otimes Z_i\in U(\frg)\otimes C(\frp).
\end{equation}
Here $C(\frp)$ is the Clifford algebra of $\frp$ defined with respect to $B(\cdot, \cdot)|_{\frp}$. One sees easily that $D$ is independent of the choice of the orthonormal basis $\{Z_i\}_{i=1}^{m}$. Moreover, $D^2$ is a natural Laplacian. Indeed,
\begin{equation}\label{D-square}
D^2=-\Omega_{\frg}\otimes 1 +\Omega_{\frk_{\Delta}} +(\|\rho_c\|^2-\|\rho\|^2) 1\otimes 1.
\end{equation}
Here $\Omega_{\frg}$ (resp., $\Omega_{\frk_{\Delta}}$) is the Casimir operator of $\frg$ (resp., $\frk_{\Delta}$), and $\frk_{\Delta}$ is a diagonal embedding of $\frk$ into $U(\frg)\otimes C(\frp)$.

Let $S_G$ be a spin module of $C(\frp)$. Let $\widetilde{K}$ be the subgroup of $K\times {\rm Pin}(\frp_0)$ consisting of the pairs $(k, s)$ such that ${\rm Ad}(k)=p(s)$, where ${\rm Ad}: K\to O(\frp_0)$ is the adjoint action, and $p: {\rm Pin}(\frp_0)\to O(\frp_0)$ is the double covering map.
Whenever $\pi$ is unitary, by using the Hermitian form on $\pi\otimes S_G$ and  \eqref{D-square}, one deduces Parthasarathy's Dirac operator inequality:
\begin{equation}\label{Dirac-inequality-original}
\|\gamma+\rho_c\|\geq \|\Lambda\|.
\end{equation}
Here $\Lambda$ is the infinitesimal character of $\pi$, and $\gamma$ is a highest weight of any $\widetilde{K}$-type of $\pi\otimes S_G$. Whenever  there exists a $\gamma$ such that \eqref{Dirac-inequality-original} fails, one concludes that $\pi$ is \emph{not} unitary.

To sharpen the Dirac inequality, in 1997, Vogan introduced Dirac cohomology of $\pi$ as
\begin{equation}\label{Dirac-cohomology}
H_D(\pi)=\ker D/\ker D \cap {\rm Im}\, D.
\end{equation}
Since the operator $D$ is invariant under the diagonal action of $K$ given by adjoint actions on $U(\frg)$ and $C(\frp)$, it follows that $\ker D$, ${\rm Im}\, D$ and thus $H_D(\pi)$ are all $\widetilde{K}$ modules.
Vogan conjectured that whenever non-zero, the Dirac cohomology $H_D(\pi)$ should reveal the infinitesimal character of $\pi$. This conjecture was proven by Huang and Pand\v zi\'c in 2002 \cite{HP}.  We regard $\frt_f^*$ as a subspace of $\frh_f^*$ by extending the linear functionals on $\frt_f$ to be zero on $\fra_f$.

\begin{thm}{\rm (Theorem 2.3 of \cite{HP})}\label{thm-HP}
Let $\pi$ be an irreducible ($\frg$, $K$) module with infinitesimal character $\Lambda$.
Assume that $H_D(\pi)\neq 0$, and let $\gamma\in\frt_f^{*}\subset\frh_f^{*}$ be the highest weight of any $\widetilde{K}$-type in it. Then $\Lambda=w(\gamma+\rho_{c})$ for some element $w\in W(\frg,\frh_f)$.
\end{thm}

The above theorem establishes Dirac cohomology as an invariant of $\pi$. And a natural problem relevant to the study of the unitary dual is: Can we classify all the irreducible unitary representations of $G$ with non-zero Dirac cohomology? Let us denote this set by $\widehat{G}^d$, and call it the \emph{Dirac series} of $G$ (after J.-S. Huang). It is worth noting that when $\pi$ is unitary, $\ker D \cap {\rm Im}\, D=\{0\}$ and \eqref{Dirac-cohomology} simplifies as $H_D(\pi)=\ker D^2$. Moreover, it is exactly the members of $\widehat{G}$ such that the Dirac inequality, which will be recalled in \eqref{Dirac-inequality}, becomes equality that consist of the Dirac series. Therefore, $\widehat{G}^d$ cuts out the extreme members of $\widehat{G}$ in the sense of Dirac inequality, and we expect the Dirac series to be interesting. For instance it contains the discrete series, $A_{\frq}(\lambda)$ modules in the good range. Moreover, Example 6.3 of \cite{DDY} suggests that Dirac series go beyond elliptic representations. Note that the latter have applications in endoscopy. In view of the research announcement by Barbasch and Pand\v zi\'c \cite{BP19}, Dirac series should have application in the theory of automorphic forms as well.

Recently, there has been progress on the classification of Dirac series. See \cite{BDW, DW20p}. The current paper is a continuation of this task on exceptional Lie groups.  Here we study the group $E_{7(-5)}$, by which we actually mean the connected simple real exceptional Lie group \texttt{E7\_q} in \texttt{atlas}. Its basic structure will be reviewed in Section \ref{sec-EVI-structure}. Here \texttt{atlas} \cite{At} is a software which can now compute many questions pertaining to $(\frg, K)$ modules. In particular, it detects whether $\pi$ is unitarizable or not based on the algorithm in \cite{ALTV}. We will give a brief account of it in Section \ref{sec-atlas}.

Let $\pi$ be an irreducible $(\frg, K)$ module with \texttt{atlas} final parameter $p=(x, \lambda, \nu)$. We say that $\pi$ is \emph{fully supported} if the KGB element $x$ has full support. That is, if $\texttt{support(x)}=[0, 1, \dots, l-1]$, where $l=\dim_{\bbC}\, \frh_f$ is the rank of $G$. Here $0, 1, \dots, l-1$ are the labels of certain simple roots (see Section \ref{sec-atlas}). Note that $\texttt{atlas}$ always counts from zero.  We say a member of $\pi\in\widehat{G}^d$ is \emph{FS-scattered} if it is fully supported. The Dirac series of $G$ will be arranged into two parts: finitely many FS-scattered representations (which can not be cohomologically induced in the way of Theorem \ref{thm-Vogan}), and finitely many \emph{strings} (which are cohomologically induced from FS-scattered representations of certain Levi subgroups).
More careful explanation will be given in Section \ref{sec-FS-scattered}. Now let us state our main result.

\begin{thm}\label{thm-EVI}
The set $\widehat{E_{7(-5)}}^d$ consists of $103$ FS-scattered representations whose spin-lowest $K$-types are all unitarily small, and $1450$ strings of representations. Moreover, each spin-lowest $K$-type of any Dirac series of $E_{7(-5)}$ occurs exactly once.
\end{thm}

The notions unitarily small $K$-type and spin-lowest $K$-type will be recalled in Section \ref{sec-EVI-structure}. Note that besides \texttt{atlas} and the finiteness algorithm \cite{D20} for classifying Dirac series, the sharpened Helgason-Johnson bound \cite{D20} is a key tool as well.

The linear group \texttt{E7\_q} studied here has a non-linear double cover  $\widetilde{G}$. There should be interesting large families of genuine Dirac series of $\widetilde{G}$. Currently, \texttt{atlas} is unable to handle non-linear groups. But it is a valuable project for future to investigate the Dirac series of $\widetilde{G}$.

The paper is organized as follows: Section 2 collects some preliminaries for the entire paper. Section \ref{sec-EVI-structure} reviews the structure of $E_{7(-5)}$. Section 4 classifies its Dirac series. Section 5 studies their Dirac indices, while Section \ref{sec-sup} singles out the special unipotent representations. Section 7 further improves the Helgason-Johnson bound. We present all the FS-scattered representations of $E_{7(-5)}$ in Section \ref{sec-appendix}.

\section{Preliminaries}\label{sec-pre}

This section collects necessary preliminaries. Note that by Theorem \ref{thm-HP}, a necessary condition for $\pi$ to have non-zero Dirac cohomology is that $\Lambda$ is \emph{real} in the sense of \cite[Definition 5.4.1]{Vog81}. That is, $\Lambda\in i \frt_{f, 0}^* + \fra_{f, 0}^*$. If not stated otherwise, we will assume that $\Lambda$ is real henceforth.

\subsection{Cohomological induction}\label{sec-coho}
Fix an element $H\in\frt_{f, 0}$ which is dominant for $\Delta^+(\frk, \frt_f)$. Let $\frl$ be the zero eigenspace of ${\rm ad}(H)$ on $\frg$, and let $\fru$ be the positive eigenspaces.
Then
\begin{equation}\label{theta-stable-parabolic-subalgebra}
\frq=\frl+\fru
\end{equation}
is a $\theta$-stable parabolic subalgebra of $\frg$. We choose a positive root system $(\Delta^+)^\prime(\frg, \frt_f)$ containing $\Delta(\fru, \frt_f)$ and $\Delta^+(\frk, \frt_f)$. Let $L$ be the normalizer of $\frq$ in $G$. Then $L$ is also connected and has a maximally compact subgroup $L\cap K$. Set
$$
\Delta^+(\frl, \frt_f)=\Delta(\frl, \frt_f) \cap (\Delta^+)^\prime(\frg, \frt_f).
$$
Let $\rho^L$ (resp., $\rho(\fru)$, $\rho^\prime$) be the half sum of roots in $\Delta^+(\frl, \frt_f)$ (resp., $\Delta^+(\fru, \frt_f)$, $(\Delta^+)^\prime(\frg, \frt_f)$). Then
$$
\rho^\prime=\rho^L +\rho(\fru).
$$

Cohomological induction is an important way of constructing representations of $G$ from representations of $L$. Let $Z$ be an irreducible $(\frl, L\cap K)$ module with infinitesimal character
$\Lambda_Z\in i\frt_{f, 0}^*$. After \cite{KV}, we say that $Z$ is in the \emph{good range} (relative to $\frq$ and $\frg$) if
\begin{equation}\label{good-range}
B(\Lambda_Z + \rho(\fru), \alpha)>0, \quad \forall \alpha\in\Delta(\fru, \frt_f).
\end{equation}
If the above strict inequalities are replaced by $\geq$, we will say that $Z$ is \emph{weakly good}. Let $\mathfrak{z}$ be the center of $\frl$. We say that $Z$ is in the \emph{fair range} if
\begin{equation}\label{fair-range}
B(\Lambda_Z  + \rho(\fru), \alpha|_{\mathfrak{z}})>0, \quad \forall \alpha\in\Delta(\fru, \frt_f).
\end{equation}
Similarly, we say that $Z$ is \emph{weakly fair} if the above strict inequalities are replaced by $\geq$.

Put $S$ as the dimension of $\fru\cap\frk$. Cohomological induction functors (also called \emph{Zuckerman functors} and \emph{Bernstein functors}) lift $Z$ to $(\frg, K)$ modules $\caL_j(Z)$ and $\caR^j(Z)$, where $j$ is a non-negative integer. We refer the reader to Chapter V of \cite{KV} for details. Here we simply recall that the most interesting case happens at the middle degree $S$. For instance, if $Z$ is weakly good, then $\caL_j(Z)$ and $\caR^j(Z)$ vanish except for the degree $S$, while $\caL_S(Z)\cong \caR^S(Z)$ is zero or irreducible with infinitesimal character $\Lambda_Z+\rho(\fru)$.   Moreover, $\caL_S(Z)$ is unitary provided $Z$ is unitary. When $Z$ is actually in the good range, we have that $\caL_S(Z)$ must be non-zero, and it is unitary if and only if $Z$ is unitary.

In the special case that $Z$ is a one dimensional unitary character $\bbC_{\lambda}$ of $L$, we denote $\caL_S(Z)$ as $A_{\frq}(\lambda)$. This module has infinitesimal character $\lambda+\rho^\prime$. As shown by Salamanca-Riba in \cite{SV}, any irreducible unitary $(\frg, K)$ module $\pi$ with a strongly regular real infinitesimal character $\Lambda$ must be isomorphic to an $A_{\frq}(\lambda)$ module. Here $\Lambda$ being \emph{strongly regular} means that
$$
B(\Lambda-\rho^\prime, \alpha)\geq 0, \quad \forall \alpha\in(\Delta^+)^\prime(\frg, \frt_f).
$$

In the next subsection, we will recall a canonical way of doing cohomological induction in the language of \texttt{atlas}.

\subsection{The \texttt{atlas} software}\label{sec-atlas}
Let us borrow some notation from \cite{ALTV}. We embark with a complex connected  simple algebraic group $G(\bbC)$ which has finite center. Let $\sigma$ be a \emph{real form} of $G(\bbC)$ so that $G=G(\bbC)^{\sigma}$. That is, $\sigma$ is an involutive antiholomorphic Lie group automorphism of $G(\bbC)$. Let $\theta$ be the involutive algebraic automorphism of $G(\bbC)$ corresponding to $\sigma$ via the Cartan theorem (see Theorem 3.2 of \cite{ALTV}).  Denote by $K(\bbC):=G(\bbC)^{\theta}$.

Let $H(\bbC)$ be a \emph{maximal torus} of $G(\bbC)$. Denote its Lie algebra by $\frh$. Choose a Borel subgroup $B(\bbC)\supset H(\bbC)$. This gives a positive root system $\Delta^+(\frg, \frh)$ which has $l$ simple roots being labeled as $0, 1, \dots, l-1$.  A \emph{KGB element} $x$ in \texttt{atlas} is a $K(\bbC)$-orbit of the Borel variety $G(\bbC)/B(\bbC)$. The support of $x$ is a subset of $[0, 1, \dots, l-1]$ and it is given by the \texttt{atlas} command \texttt{support(x)}.

The Langlands parameter $p=(x, \lambda, \nu)$ of an irreducible $(\frg, K)$ module $\pi$ in \texttt{atlas} consists of three parts. Besides a KGB element $x$,  the other components $\lambda\in X^*+\rho$ and $\nu\in(X^*)^{-\theta}\otimes_{\bbZ} \bbC$. Here $X^*$ is the group of algebraic homomorphisms
from $H(\bbC)$ to $\bbC^{\times}$. Now a representative of the infinitesimal character of $\pi$ can be chosen as
\begin{equation}\label{inf-char}
\frac{1}{2}(1+\theta)\lambda +\nu \in\frh^*.
\end{equation}
The Cartan involution $\theta$ now becomes $\theta_x$---the involution of $x$, and the latter is given by the command \texttt{involution(x)}. All the (finitely many) irreducible $(\frg, K)$ modules with infinitsimal character \texttt{Lambda} can be constructed via the command
\begin{verbatim}
set all=all_parameters_gamma(G, Lambda)
\end{verbatim}

It is worth mentioning that any irreducible $(\frg, K)$ module $\pi$ can be realized as a cohomologically induced module from the weakly good range in the following canonical way.

\begin{thm}\label{thm-Vogan} \emph{(Vogan \cite{Vog84})}
 Let $p=(x, \lambda, \nu)$ be the \texttt{atlas} final parameter of an irreducible $(\frg, K)$-module $\pi$.
Let $S(x)$ be the support of $x$, and $\frq(x)$ be the $\theta$-stable parabolic subalgebra given by the pair $(S(x), x)$, with Levi factor $L(x)$.
Then $\pi$ is cohomologically induced, in the weakly good range,
from an irreducible $(\frl, L(x)\cap K)$-module $\pi_{L(x)}$ with final parameter $p_{L(x)}=(y, \lambda-\rho(\fru), \nu)$, where $y$ is the KGB element of $L(x)$ corresponding to the KGB element $x$ of $G$.
\end{thm}

The above form was phrased by Paul \cite{Paul}. It can be realized in \texttt{atlas} as follows:
\begin{verbatim}
set p=parameter(x, lambda, nu)
set (Q, q):=reduce_good_range(p)
set L=Levi(Q)
\end{verbatim}
Here \texttt{q} is the inducing module for \texttt{p}, and \texttt{L} is the Levi factor $L(x)$. A special case is that $x$ is fully supported, i.e., $S(x)=[0, 1, \dots, l-1]$. Then $L(x)=G$ and $\frq(x)=\frg$. In other words, there is no induction in this case.

\subsection{Organizing the Dirac series}\label{sec-FS-scattered}
Theorem \ref{thm-Vogan} allows us to organize the irreducible $(\frg, K)$ modules with a fixed infinitesimal character according to their supports. Then any $\pi$ is either fully supported, or cohomologically induced within the weakly good range from a fully supported module on the Levi level.

It also hints us to organize the Dirac series of $G$ in a similar way. As shown by Theorem A of \cite{D20}, the FS-scattered representations must be finitely many, and there is an algorithm to pin down them. Once some conditions have been verified for $G$ (see Section \ref{sec-EVI-esc}), those Dirac series which are \emph{not} fully supported can be organized into finitely many strings, with each string being cohomologically induced from a FS-scattered representation of certain $L(x)$ tensored with its unitary characters.

\section{Basic structure of $E_{7(-5)}$}\label{sec-EVI-structure}

From now on, we fix $G$ as the simple real exceptional linear Lie group \texttt{E7\_q} in \texttt{atlas}. This connected group is equal rank. That is, $\frh_f=\frt_f$. It has center $\bbZ/2\bbZ$. The Lie algebra $\frg_0$ of $G$ is denoted as \texttt{EVI} in \cite[Appendix C]{Kn}.
Note that
$$
-\dim \frk +\dim \frp=-69+64=-5.
$$
Therefore, the group $G$ is also called $E_{7(-5)}$ in the literature.

We present a Vogan diagram for $\frg_0$ in Fig.~\ref{Fig-EVI-Vogan}, where $\alpha_1=\frac{1}{2}(1, -1,-1,-1,-1,-1,-1,1)$, $\alpha_2=e_1+e_2$ and $\alpha_i=e_{i-1}-e_{i-2}$ for $3\leq i\leq 7$. The black dot means that $\alpha_1$ is non-compact, while all the other simple roots are compact. By specifying a Vogan diagram, we have actually fixed a choice of positive roots $\Delta^+(\frg, \frt_f)$.  Let $\zeta_1, \dots, \zeta_7\in\frt_f^*$ be the corresponding fundamental weights for $\Delta^+(\frg, \frt_f)$. The dual space $\frt_f^*$ will be identified with $\frt_f$ under the Killing form $B(\cdot, \cdot)$.
We will use $\{\zeta_1, \dots, \zeta_7\}$ as a basis to express the \texttt{atlas} parameters $\lambda$, $\nu$ and the infinitesimal character $\Lambda$. More precisely, in such cases, $[a, b, c, d, e, f, g]$ will stand for the vector $a\zeta_1+\cdots+ g \zeta_7$.

\begin{figure}[H]
\centering
\scalebox{0.6}{\includegraphics{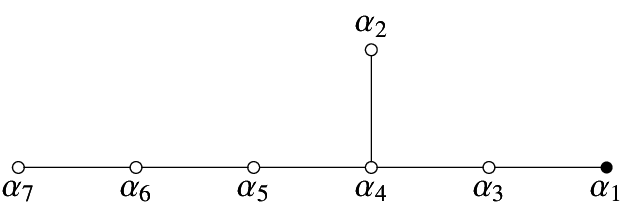}}
\caption{The Vogan diagram for EVI}
\label{Fig-EVI-Vogan}
\end{figure}

Put $\gamma_i=\alpha_{8-i}$ for $1\leq i\leq 4$, $\gamma_5=\alpha_2$, $\gamma_6=\alpha_3$ and
\begin{equation}\label{gamma-7}
\gamma_7=2\alpha_1+ 2\alpha_2+ 3\alpha_3+ 4\alpha_4+ 3\alpha_5 + 2\alpha_6+\alpha_7=
(0, 0, 0, 0, 0, 0, -1, 1),
\end{equation}
which is the highest root of $\Delta^+(\frg, \frt_f)$.
Then $\gamma_1, \dots, \gamma_7$ are the simple roots of $\Delta^+(\frk, \frt_f)=\Delta(\frk, \frt_f)\cap \Delta^+(\frg, \frt_f)$. We present the Dynkin diagram of $\Delta^+(\frk, \frt_f)$  in Fig.~\ref{Fig-EVI-K-Dynkin}.
Let $\varpi_1, \dots, \varpi_7\in \frt_f^*$ be the corresponding fundamental weights. Note that
\begin{equation}\label{rhoc}
\rho_c=-7\zeta_1+\zeta_2+\cdots+\zeta_7.
\end{equation}

\begin{figure}[H]
\centering
\scalebox{0.6}{\includegraphics{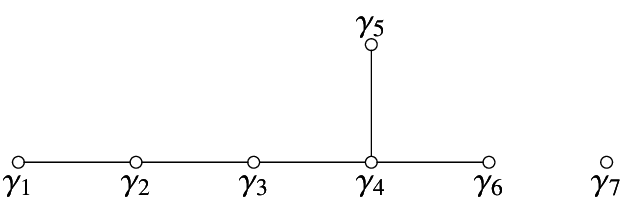}}
\caption{The Dynkin diagram for $\Delta^+(\frk, \frt_f)$}
\label{Fig-EVI-K-Dynkin}
\end{figure}

\subsection{$\frk$-types and $K$-types}\label{sec-k-type-K-type}
Let $E_{\mu}$ be the $\frk$-type with highest weight $\mu$.
We will use $\{\varpi_1, \dots, \varpi_7\}$ as a basis to express $\mu$.
Namely, in such a case, $[a, b, c, d, e, f, g]$ stands for the vector $a \varpi_1+b \varpi_2 + c \varpi_3+d \varpi_4 +e \varpi_5+ f\varpi_6+ g\varpi_7$. For instance,
\begin{equation}\label{beta}
\beta:=\alpha_1+ 2\alpha_2+ 3\alpha_3+ 4\alpha_4+ 3\alpha_5 + 2 \alpha_6 +\alpha_7=[0, 0, 0, 0, 0, 1, 1]
\end{equation}
and $\dim \frp=\dim E_{\beta}=64$. The $\frk$-type $E_{[a, b, c, d, e, f, g]}$ is self-dual. For $a, b,c ,d ,e, f, g\in\bbZ_{\geq 0}$, we have that $E_{[a, b, c, d, e, f, g]}$ is  a $K$-type if and only if
\begin{equation}\label{EVI-K-type}
a+c+f+g \mbox{ is even.}
\end{equation}

Shifting the coordinates of a $K$-type from the \texttt{atlas} fashion to the current one is a linear algebra question. We learned from Vogan that there is a clever choice of the KGB element which makes the job much easier.  Let us illustrate this for \texttt{E7\_q} (certain outputs are omitted).
\begin{verbatim}
G:E7_q
void:for x in distinguished_fiber(G) do prints(x," ",rho_c(x)) od
KGB element #0 [ -1,  1,  2, -1,  2, -1,  2 ]/1
KGB element #1 [  1,  1,  1, -1,  2, -1,  2 ]/1
......
KGB element #61 [  7,  1, -6,  1,  1,  1,  1 ]/1
KGB element #62 [ -7,  1,  1,  1,  1,  1,  1 ]/1
\end{verbatim}
Note that for \texttt{KGB(G, 62)}, the coordinates agree with those in \eqref{rhoc}. It is this KGB element that we will choose to look at the $K$-types.  See Example \ref{example-EVI-Phi8}. In particular, the linear algebra question will be answered in \eqref{atlas-shift-our}.

A final remark is that the number $62$ equals $63-1$. Recall that \texttt{atlas} counts from zero and that $|W(\frg, \frt_f)^1|=63$ (see the next subsection). This should not be an accident.

\subsection{Positive root systems and cones}
Via choosing a Vogan diagram of $\frg_0$ in Fig.~\ref{Fig-EVI-Vogan}, we have actually fixed a positive root system $\Delta^+(\frg, \frh_f)$, where $\frh_f=\frt_f$ for $G$. Denote by
$$
(\Delta^+)^{(0)}(\frg, \frt_f):=\Delta^+(\frg, \frt_f).
$$
It contains the positive root system $\Delta^+(\frk, \frt_f)$ of $\Delta(\frk, \frt_f)$ corresponding to Fig.~\ref{Fig-EVI-K-Dynkin}. We fix this $\Delta^+(\frk, \frt_f)$ once for all. Then there are $63$ positive root systems of $\Delta(\frg, \frt_f)$ containing the $\Delta^+(\frk, \frt_f)$. We enumerate them as follows:
$$
(\Delta^+)^{(j)}(\frg, \frt_f)=\Delta^+(\frk, \frt_f) \cup  (\Delta^+)^{(j)}(\frp, \frt_f), \quad 0\leq j\leq 62.
$$
For $0\leq j\leq 62$, denote by $\rho_n^{(j)}$ (resp., $\rho^{(j)}$) the half sum of roots in $(\Delta^+)^{(j)}(\frp, \frt_f)$ (resp., $(\Delta^+)^{(j)}(\frg, \frt_f)$)
and let $\pi$ be an irreducible $(\frg, K)$ module. Then of course
$$
\rho^{(j)}=\rho_n^{(j)}+\rho_c,
$$
where $\rho_c$ is the half sum of roots in $\Delta^+(\frk, \frt_f)$. Let $w^{(j)}$ be the unique element in $W(\frg, \frt_f)$ such that $w^{(j)} \rho^{(0)}=\rho^{(j)}$. Note that $w^{(0)}=e$. Denote by
$$
W(\frg, \frt)^1=\{w^{(j)}\mid 0\leq j\leq 62\}.
$$
By a result of Kostant \cite{Ko}, the multiplication map gives a bijection from $W(\frg, \frt)^1\times W(\frk, \frt_f)$ onto $W(\frg, \frt_f)$. In particular, this explains that
$$
63=|W(\frg, \frt_f)^1|=\frac{|W(\frg, \frt_f)|}{|W(\frk, \frt_f)|}=\frac{2903040}{46080}.
$$

Note that $w^{(j)}\alpha_1, \dots, w^{(j)}\alpha_7$ are the simple roots of $(\Delta^+)^{(j)}(\frg, \frt_f)$ for $0\leq j\leq 62$, and $w^{(j)}\zeta_1, \dots, w^{(j)}\zeta_7$ are the corresponding fundamental weights. We denote by $\caC^{(j)}$ (resp., $\caC$) the dominant Weyl chamber for $\Delta^+(\frg, \frt_f)$ (resp., $\Delta^+(\frk, \frt_f)$). Then $w^{(j)}\caC^{(0)}=\caC^{(j)}$ and
\begin{equation}\label{cone-deco}
\caC=\bigcup_{j=0}^{62}\caC^{(j)}.
\end{equation}

\subsection{The u-small convex hull and spin module}
The weights $\rho_n^{(j)}$ are quite important. On one hand, the convex hull formed by the $W(\frk, \frt_f)$ orbits of the weights $2\rho_n^{(j)}$, $0\leq j\leq 62$, is called the \emph{unitarily small} (\emph{u-small} for short henceforth) \emph{convex hull} by Salamanca-Riba and Vogan \cite{SV}. A $K$-type (or $\frk$-type) is called \emph{u-small} if its highest weight lies in this polyhedron. Otherwise, we will say it is \emph{u-large}.

On the other hand, since $\dim \frp=64$ is even, the Clifford algebra $C(\frp)$ has a unique irreducible module, namely, the spin module $S_G$.  By Lemma 6.9 of \cite[Chapter II]{BW}, we have that
\begin{equation}\label{spin-module}
S_G=\bigoplus_{j=0}^{62} E_{\rho_n^{(j)}}
\end{equation}
as $\frk$ modules.

\subsection{Lambda norm, LKTs and infinitesimal character}

The lambda norm of a $K$-type $\mu$ was introduced by Vogan \cite{Vog81}. Choose an index $0\leq j\leq 62$ such that $\mu+2\rho_c\in \caC^{(j)}$.
Put
\begin{equation}\label{lambda-a-mu}
\lambda_a(\mu):=P(\mu+2 \rho_c -\rho^{(j)}).
\end{equation}
Here  $P(\mu+2 \rho_c -\rho^{(j)})$ is the projection of $\mu+2 \rho_c -\rho^{(j)}$ onto the cone $\caC^{(j)}$, i.e., it denotes the unique point in the cone $\caC^{(j)}$ which is closest to the point $\mu+2 \rho_c -\rho^{(j)}$. It turns out that $\lambda_a(\mu)$ is independent of the choice of an allowable index $j$. The \emph{lambda norm} of $\mu$ is defined as
\begin{equation}\label{lambda-norm}
\|\mu\|_{\rm lambda}:=\|\lambda_a(\mu)\|.
\end{equation}
The above geometric way of equivalently describing the lambda norm in  \cite{Vog81} is due to Carmona \cite{Ca}.

Now let $\pi$ be an irreducible $(\frg, K)$ module. A $K$-type $\mu$ is called a \emph{lowest $K$-type} (LKT for short) of $\pi$ if $\mu$ occurs in $\pi$ and $\|\mu\|_{\rm lambda}$ attains the minimum among all the $K$-types of $\pi$. It is easy to see that $\pi$ has finitely many LKTs.

Being important invariants of $\pi$, its infinitesimal character and LKTs are linked in the following way: Let $\mu$ be one of the LKTs $\mu$ of $\pi$, then a representative of the infinitesimal character of $\pi$ can be chosen as
\begin{equation}\label{inf-char}
\Lambda=(\lambda_a(\mu), \nu)\in \frh^*=\frt^*+\fra^*.
\end{equation}
Here, as in \cite{Vog81}, $G(\lambda_a(\mu))$ is the isotropy group at $\lambda_a(\mu)$ for the $G$ action; $\frh$ is the complexified Lie algebra of a maximally split $\theta$-stable Cartan subgroup $H=TA$ of $G(\lambda_a(\mu))$. Note that $\nu$ in \eqref{inf-char} has the same norm as the $\nu$-part in the \texttt{atlas} final parameter $p=(x, \lambda, \nu)$ of $\pi$. We will not distinguish them by abusing the notation a bit.

\subsection{Spin norm, spin LKTs and Vogan pencil}
Inspired by Vogan's lambda norm, the second named author introduced spin norm in his 2011 HKUST thesis. See \cite{D13}. In the current setting, the \emph{spin norm}  of the $K$-type $E_{\mu}$ specializes as
\begin{equation}\label{spin-norm}
\|\mu\|_{\rm spin} =\min_{0\leq j \leq 62}\|\{\mu-\rho_n^{(j)}\}+\rho_c\|.
\end{equation}
Note that $E_{\{\mu-\rho_n^{(j)}\}}$ is the PRV component \cite{PRV} of the tensor product of $E_{\mu}$ with the contragredient $\frk$-type of $E_{\rho_n^{(j)}}$ (which is just $E_{\rho_n^{(j)}}$ in the current setting).

The spin norm of an irreducible $(\frg, K)$ module $\pi$ is defined as
$$
\|\pi\|_{\rm spin}=\min \|\delta\|_{\rm spin},
$$
where $\delta$ runs over all the $K$-types of $\pi$.
If $\delta$ attains $\|\pi\|_{\rm spin}$, we will call it a \emph{spin-lowest $K$-type} of $\pi$. Again, it is easy to see that $\pi$ has finitely many spin LKTs.

If $\pi$ is further assumed to be unitary,
the original Dirac operator inequality \eqref{Dirac-inequality-original} can be rephrased as
\begin{equation}\label{Dirac-inequality}
\|\pi\|_{\rm spin}\geq \|\Lambda\|.
\end{equation}
Moreover, as shown in \cite{D13}, the equality happens in \eqref{Dirac-inequality} if and only if $\pi$ has non-zero Dirac cohomology, and in this case, it is exactly the spin-lowest $K$-types of $\pi$ that contribute to $H_D(\pi)$.

Since \texttt{E7\_q} is not of Hermitian symmetric type, Lemma 3.4 of \cite{Vog80} says that the $K$-type $V_{\delta + n\beta}$ must show up in $\pi$ for any non-negative integer $n$ provided that $V_{\delta}$ occurs in $\pi$. We call them the \emph{Vogan pencil} starting from $V_{\delta}$.   Now it follows from \eqref{Dirac-inequality} that
\begin{equation}\label{Dirac-inequality-improved}
\|\delta+n\beta\|_{\rm spin}\geq \|\Lambda\|, \quad \forall n\in \bbZ_{\geq 0}.
\end{equation}
In other words, whenever \eqref{Dirac-inequality-improved} fails, we can conclude that $\pi$ is non-unitary. Distribution of spin norm along Vogan pencils has been discussed in Theorem C of \cite{D17}. In practice,
we will take $\delta$ to be a lowest $K$-type of $\pi$ and use the corresponding Vogan pencil to do non-unitarity test.

\section{Dirac series of $E_{7(-5)}$}\label{sec-Dirac-series}

This section aims to classify the Dirac series of $E_{7(-5)}$. Besides \texttt{atlas}, our main tools include the finiteness algorithm reported in \cite{D17}, the sharpened Helgason-Johnson bound \cite{D20}, and the way of counting strings in the Dirac series \cite{D21}.

\subsection{FS-scattered representations of $E_{7(-5)}$}\label{sec-FS-EII}

This subsection aims to sieve out all the FS-scattered Dirac series representations for $E_{7(-5)}$ using the algorithm in \cite{D17}.

\begin{lemma}\label{lemma-EVI-HP}
Let $\Lambda=a\zeta_1+b\zeta_2+c\zeta_3+d\zeta_4+e\zeta_5+f \zeta_6+g \zeta_7$ be the infinitesimal character of any Dirac series representation $\pi$ of $E_{7(-5)}$ which is dominant with respect to $\Delta^+(\frg, \frt_f)$. Then $a$, $b$, $c$, $d$, $e$, $f$, $g$ must be non-negative integers such that $a+c>0$, $b+d>0$, $c+d>0$, $d+e>0$, $e+f>0$, $f+g>0$ and $b+e+g>0$.
\end{lemma}
\begin{proof}
Since the group \texttt{E7\_q} is linear, it follows from Remark 4.1 of \cite{D21} that $a$, $b$, $c$, $d$, $e$, $f$, $g$ must be non-negative integers.

Now if $a+c=0$, i.e., $a=c=0$,  a direct check says that for any $w\in W(\frg, \frt_f)^1$, at least one coordinate of $w\Lambda$ in terms of the basis $\{\varpi_1, \dots, \varpi_7\}$ vanishes. Therefore,
$$
\{\mu-\rho_n^{(j)}\} + \rho_c =w \Lambda
$$
could not hold for any $K$-type $\mu$. This proves that $a+c>0$. Other inequalities can be similarly deduced.
\end{proof}

To obtain all the FS-scattered Dirac series representations of $E_{7(-5)}$, now it suffices to consider all the infinitesimal characters $\Lambda=[a, b, c, d, e, f, g]$ such that
\begin{itemize}
\item[$\bullet$] $a$, $b$, $c$, $d$, $e$, $f$, $g$ are non-negative integers;
\item[$\bullet$] $a+c>0$, $b+d>0$, $c+d>0$, $d+e>0$, $e+f>0$, $f+g>0$, $b+e+g>0$;
\item[$\bullet$] $\min\{a, b, c, d, e, f, g\}=0$;
\item[$\bullet$] there exists a fully supported KGB element $x$ such that $\|\frac{\Lambda-\theta_x\Lambda}{2}\|\leq \sqrt{\frac{227}{2}}$.
\end{itemize}
Let us collect them as $\Phi$. Note that the first item and the second item are guaranteed by Lemma \ref{lemma-EVI-HP}. The third item uses the main result of  Salamanca-Riba \cite{Sa} which is recalled at the end of Section \ref{sec-coho}.  For the fourth item above, we note that if $\texttt{p}$ is an irreducible representation in $\texttt{atlas}$ with final parameter $(x, \lambda, \nu)$ and with infinitesimal character $\Lambda$ which is dominant, then
\begin{equation}\label{nu}
\nu=\frac{\Lambda-\theta_x(\Lambda)}{2},
\end{equation}
where $\theta_x$ is \texttt{involution(x)} in \texttt{atlas}. Therefore,  if the representation \texttt{p} is infinite-dimensional and unitary, by \cite{D20}, we must have
\begin{equation}\label{new-HJ-bound}
\|\nu\|\leq \sqrt{\frac{227}{2}}.
\end{equation}
This bound is attained on the first entry of Table \ref{table-EVI-1110111}, which is the minimal representation.
On the other hand, the original Helgason-Johnson upper bound for $\|\nu\|$ is $\|\rho\|=\sqrt{\frac{399}{2}}$ (see \cite{HJ}). This improvement greatly reduces the cardinality of $\Phi$, which turns out to be $723855$.
As put in Remark \ref{rmk-example-EVI-Phi8} below, the sharpened Helgason-Johnson bound will continue to be helpful in doing non-unitarity test.

Let us collect all the members of $\Phi$ whose largest coordinate equals to $i$ as $\Phi_i$. Then $\Phi$ is partitioned into $\Phi_1, \dots, \Phi_{14}$. More precisely, we have the following

\begin{center}
\begin{tabular}{c|c|c|c|c|c|c}
$\#\Phi_1$ & $\#\Phi_2$ & $\#\Phi_3$ & $\#\Phi_4$ & $\#\Phi_5$ & $\#\Phi_6$ & $\#\Phi_7$  \\
\hline
$32$ & $1056$  & $8470$ & $35765$ & $87593$ & $132922$ & $142493$\\
\hline
$\#\Phi_8$  & $\#\Phi_9$ & $\#\Phi_{10}$& $\#\Phi_{11}$ & $\#\Phi_{12}$ & $\#\Phi_{13}$ & $\#\Phi_{14}$ \\
\hline
$122007$ & $89151$ & $55349$ & $30432$ & $13410$ & $4494$ & $681$
\end{tabular}
\end{center}
The elements of $\Phi_1$ are listed as follows:
\begin{align*}
&[0, 0, 1, 1, 0, 1, 1], [0, 0, 1, 1, 1, 0, 1], [0, 0, 1, 1, 1, 1, 0], [0, 0, 1, 1, 1, 1, 1], \\
&[0, 1, 1, 0, 1, 0, 1], [0, 1, 1, 0, 1, 1, 0], [0, 1, 1, 0, 1, 1, 1], [0, 1, 1, 1, 0, 1, 0], \\
&[0, 1, 1, 1, 0, 1, 1], [0, 1, 1, 1, 1, 0, 1], [0, 1, 1, 1, 1, 1, 0], [0, 1, 1, 1, 1, 1, 1], \\
&[1, 0, 0, 1, 0, 1, 1], [1, 0, 0, 1, 1, 0, 1], [1, 0, 0, 1, 1, 1, 0], [1, 0, 0, 1, 1, 1, 1], \\
&[1, 0, 1, 1, 0, 1, 1], [1, 0, 1, 1, 1, 0, 1], [1, 0, 1, 1, 1, 1, 0], [1, 0, 1, 1, 1, 1, 1], \\
&[1, 1, 0, 1, 0, 1, 0], [1, 1, 0, 1, 0, 1, 1], [1, 1, 0, 1, 1, 0, 1], [1, 1, 0, 1, 1, 1, 0], \\
&[1, 1, 0, 1, 1, 1, 1], [1, 1, 1, 0, 1, 0, 1], [1, 1, 1, 0, 1, 1, 0], [1, 1, 1, 0, 1, 1, 1], \\
&[1, 1, 1, 1, 0, 1, 0], [1, 1, 1, 1, 0, 1, 1], [1, 1, 1, 1, 1, 0, 1], [1, 1, 1, 1, 1, 1, 0].
\end{align*}
A careful study of the irreducible unitary representations under the above $32$ infinitesimal characters leads us to Section \ref{sec-appendix}. Although $\Phi_1$ only occupies a very small portion (about 0.0044\%) of $\Phi$, for us the most interesting story turns out to happen \emph{within} $\Phi_1$.

Indeed, let $\Pi_{\rm FS}(\Lambda)$ (resp., $\Pi_{\rm FS}^{\rm u}(\Lambda)$) be the set of all the fully supported irreducible (resp., unitary) representations with infinitesimal character $\Lambda$.
Conjecture 1.4 of \cite{DDH} says that $\Pi_{\rm FS}^{\rm u}(\Lambda)$ should be empty for any $\Lambda\in\Phi$ which has a coordinate bigger than $1$.

\begin{lemma}\label{lemma-b2-empty}
Let $G$ be $E_{7(-5)}$. Then $\Pi_{\rm FS}^{\rm u}(\Lambda)$ is empty for any $\Lambda\in\Phi_i$ for $2\leq i\leq 14$.
\end{lemma}

The lemma is actually obtained by tedious calculations, whose explanation will occupy the remaining part of this subsection. Currently, these calculations seem to be almost impossible for exceptional groups with rank $8$. Thus a more clever way will be appreciated. For this, the last few minutes of \cite{Vog22} is inspirational.

To arrive at Lemma \ref{lemma-b2-empty}, for the elements $\Lambda$ in $\Phi_i$, $2\leq i\leq 14$, we will mainly use Parthasarathy's Dirac operator inequality, and distribution of spin norm along the Vogan pencil starting from one lowest $K$-type to verify that $\Pi_{\rm FS}^{\rm u}(\Lambda)$ is empty. This method turns out to be very effective in non-unitarity test.

\begin{example}\label{example-EVI-Phi8}
Let us consider the infinitesimal character $\Lambda=[1, 1, 1, 1, 1, 8, 0]$ in $\Phi_8$.
\begin{verbatim}
G:E7_q
set all=all_parameters_gamma(G, [1, 1, 1, 1, 1, 8, 0])
#all
Value: 4668
set allFS=## for p in all do if #support(p)=7 then [p] else [] fi od
#allFS
Value: 2423
\end{verbatim}
Therefore, there are $4668$ irreducible representations under $\Lambda$, while $|\Pi_{\rm FS}(\Lambda)|=2423$.

Now let us compare the original Helgason-Johnson bound and the sharpened one. Since it is easier to enter integers (rather than rational numbers) into \texttt{atlas}, we collect the transposed $2 \zeta_1, 2\zeta_2, \dots, 2\zeta_7$ as \texttt{TgFWts}. Then the command
\begin{verbatim}
(nu(p)*TgFWts)*(nu(p)*TgFWts)<=4*399/2}
\end{verbatim}
below is equivalent to $\|\nu\|\leq\sqrt{\frac{399}{2}}$.
\begin{verbatim}
set HJ=227/2
set TgFWts=mat: [[0, 1, -1, 0, 0, 0, 0], [0, 1, 1, 0, 0, 0, 0],
[0, 1, 1, 2, 0, 0, 0], [0, 1, 1, 2, 2, 0, 0], [0, 1, 1, 2, 2, 2, 0],
[0, 1, 1, 2, 2, 2, 2], [-2, -2, -3, -4, -3, -2, -1], [2, 2, 3, 4, 3, 2, 1]]
set oldHJ=##for p in allFS do if (nu(p)*TgFWts)*(nu(p)*TgFWts)
<=4*399/2 then [p] else [] fi od
#oldHJ
Value: 214
set newHJ=##for p in allFS do if (nu(p)*TgFWts)*(nu(p)*TgFWts)
<=4*HJ then [p] else [] fi od
#newHJ
Value: 4
\end{verbatim}
Therefore, there are $214$ fully  supported irreducible representations satisfying $\|\nu\|\leq \sqrt{\frac{399}{2}}$, while only $4$ of them further satisfies $\|\nu\|\leq \sqrt{\frac{227}{2}}$.

Finally, let us look at one of the LKTs of the above four representations.
\begin{verbatim}
set x62=KGB(G,62)
void: for p in newHJ do print(highest_weight(LKTs(p)[0],x62)) od
((),KGB element #62,[ -9,  1,  1,  5,  1,  5,  1 ])
((),KGB element #62,[ -8,  2,  0,  4,  1,  6,  2 ])
((),KGB element #62,[ -10, 2,  2,  2,  1,  8,  2 ])
((),KGB element #62,[ -8,  1,  0,  6,  0,  5,  2 ])
\end{verbatim}
Here we note that when a $K$ type has coordinate $y_1, y_2, \dots, y_7$ using the KGB element \texttt{x62} above, then its highest weight in terms of $\varpi_1, \dots, \varpi_7$ (see Section \ref{sec-k-type-K-type}) is
\begin{equation}\label{atlas-shift-our}
[y_7, y_6, y_5, y_4, y_2, y_3, 2 y_1 + 2 y_2 + 3 y_3 + 4 y_4 + 3 y_5 + 2 y_6 + y_7].
\end{equation}
For instance, the LKTs are
$$
[1, 5, 1, 5, 1, 1, 21], \quad [2, 6, 1, 4, 2, 0, 21], \quad [2, 8, 1, 2, 2, 2, 19], \quad [2, 5, 0, 6, 1, 0, 22].
$$
The minimum spin norm along the Vogan pencils starting from them are
$$
\sqrt{594}, \sqrt{596}, \sqrt{600}, \sqrt{606}
$$
respectively, while $\|\Lambda\|=\sqrt{706}$. Therefore, Dirac inequality says that none of them is unitary, i.e., $\Pi_{\rm FS}^{\rm u}([1, 1, 1, 1, 1, 8, 0])=\emptyset$.
\hfill\qed
\end{example}
\begin{rmk}\label{rmk-example-EVI-Phi8}
Since \texttt{atlas} takes some time to calculate the LKTs, adopting the sharpened Helason-Johnson bound (thus reducing the number of representations for which we need to print their LKTs) saves us a lot of time.
\end{rmk}

The above method fails on $60$ members of $\Phi$:
$$
[0,0,1,1,0,1,2], [0,0,1,1,0,2,1], [0,0,1,1,1,0,2],
\dots, [3,1,0,1,0,1,0], [1, 1, 0, 1, 0, 1, 4].
$$
However, a more careful look says that there is no fully supported irreducible unitary representation under them. Let us provide one example.

\begin{example}
Consider the infinitesimal character
$\Lambda=[0, 0, 1, 1, 1, 0, 2]$ for $\texttt{E7\_q}$.

\begin{verbatim}
G:E7_q
set all=all_parameters_gamma(G,[0,0,1,1,1,0,2])
#all
Value: 1085
set allFS=## for p in all do if #support(p)=7 then [p] else [] fi od
#allFS
Value: 609
\end{verbatim}
Therefore $|\Pi_{\rm FS}(\Lambda)|=609$.

The following suggests that the bound $227/2$ does not work for \texttt{allFS}:
\begin{verbatim}
set HJ=227/2
set newHJ=##for p in allFS do if (nu(p)*TgFWts)*(nu(p)*TgFWts)
<=4*HJ then [p] else [] fi od
#newHJ
Value: 609
\end{verbatim}
The \texttt{TgFWts} involved above is the same as in Example \ref{example-EVI-Phi8} and is thus omitted.

A careful look at \texttt{newHJ} says that the non-unitarity test using the pencil starting from one LKT fails only for the representations \texttt{newHJ[187]} and \texttt{newHJ[294]}. Indeed, \texttt{newHJ[187]} has a unique LKT $\delta=[0, 1, 0, 0, 4, 0, 0]$, and the minimum spin norm along the pencil $\{\delta+n\beta\mid n\in\bbZ_{\geq 0}\}$ is $\sqrt{\frac{223}{2}}$, while $\|\Lambda\|=\sqrt{\frac{215}{2}}$. The representation \texttt{newHJ[294]} has a unique LKT $\delta^\prime=[6, 0, 0, 0, 0, 0, 10]$, and the minimum spin norm along the pencil $\{\delta^\prime+n\beta\mid n\in\bbZ_{\geq 0}\}$ is $\sqrt{\frac{231}{2}}$. Thus Dirac inequality does not work for them. Instead, we check their unitarity directly:
\begin{verbatim}
is_unitary(newHJ[187])
Value: false
is_unitary(newHJ[294])
Value: false
\end{verbatim}
Thus $\Pi_{\rm FS}^{\rm u}(\Lambda)=\emptyset$.
\hfill\qed
\end{example}

\subsection{Counting the strings in $\widehat{E_{7(-5)}}^d$}\label{sec-EVI-esc}

Firstly, let us verify that Conjecture 2.6 of \cite{D21} and the binary condition hold for $E_{7(-5)}$.

\begin{example}
Consider the case that $\texttt{support(x)=[1, 2, 3, 4, 5, 6]}$. There are $518$ such KGB elements in total. We compute that there are $138546$ infinitesimal characters $\Lambda=[a, b, c, d, e, f, g]$ in total such that
\begin{itemize}
\item[$\bullet$] $b$, $c$, $d$, $e$, $f$, $g$ are non-negative integers, $a=0$;
\item[$\bullet$] $a+c>0$, $b+d>0$, $c+d>0$, $d+e>0$, $e+f>0$, $f+g>0$ and $b+e+g>0$;
\item[$\bullet$] there exists a  KGB element $x$ with support $[1, 2, 3, 4, 5, 6]$ such that $\|\frac{\Lambda-\theta_x\Lambda}{2}\|\leq \sqrt{\frac{227}{2}} $.
\end{itemize}
As in Section \ref{sec-FS-EII}, we exhaust all the irreducible unitary representations under these infinitesimal characters  with the above $518$ KGB elements. It turns out that such representations occur only when $b, c, d, e, f, g=0$ or $1$.  Then we check that each $\pi_{L(x)}$ is indeed unitary. \hfill\qed
\end{example}

All the other non fully supported KGB elements are handled similarly. Thus Conjecture 2.6 of \cite{D21} and the binary condition hold for $E_{7(-5)}$.

Now we use the formula in Section 5 of \cite{D21} to pin down the number of strings in $\widehat{E_{7(-5)}}^d$.
We compute that
\begin{align*}
&N([0,1,2,4,5,6])=1, \quad N([0,1,2,3,5,6])=8, \quad N([0,1,3,4,5,6])=11, \\
&N([0,1,2,3,4,6])=15, \quad N([0,2,3,4,5,6])=54, \quad N([1,2,3,4,5,6])=83,\\
&N([0,1,2,3,4,5])=97.
\end{align*}
In particular, it follows that $N_6=259$. We also compute that
$$
N_0=63, \quad N_1=112, \quad N_2=156, \quad N_3=226, \quad N_4=300, \quad N_5=334.
$$
Therefore, the total number of strings for $E_{7(-5)}$ is equal to
$$
\sum_{i=0}^{6} N_i=1450.
$$

To end up with this section, we mention that some auxiliary files have been built up to facilitate the classification of the Dirac series of $E_{7(-5)}$. They are available via the following link:
\begin{verbatim}
https://www.researchgate.net/publication/360060854_EVI-Files
\end{verbatim}

\section{Dirac index}\label{sec-DI}

Let $\pi$ be an irreducible unitary $(\frg, K)$ module. Let $H_D^+(\pi)$ (resp., $H_D^-(\pi)$) denotes the even (resp., odd) part of the Dirac cohomology $H_D(\pi)$. Then the \emph{Dirac index} of $\pi$ is defined as the virtual $\widetilde{K}$ module
\begin{equation}
{\rm DI}(\pi)=H_D^+(\pi) - H_D^-(\pi).
\end{equation}

The definition suggests a direct way of computing the Dirac index from knowledge of the Dirac cohomology: for any $\widetilde{K}$ type $\gamma$ living in $H_D(\pi)$, it remains to detect whether $\gamma$ lives in $\pi\otimes S^+_G$ or $\pi\otimes S^-_G$. The latter  is further given by the parity of $l(w^{(j)})$, where $\rho_n^{(j)}:=w^{(j)}\rho-\rho_c$, $0\leq j\leq s-1$, are the components of $S_G$ as $\frk$ types (see Lemma 2.3 of \cite{DW21}). A more efficient way of computing Dirac index is given by \cite{MPVZ}. Its \texttt{atlas} realization is the following command:
\begin{verbatim}
show_Dirac_index(p)
\end{verbatim}

The first $\pi$ for which cancellation happens between $H_D^+(\pi)$ and $H_D^-(\pi)$ is found on the group \texttt{F4\_s} \cite{DDY}. Later, more such representations are reported on \texttt{E6\_q} \cite{D21}.
Currently, for \texttt{E7\_q}, there are $13$ FS-scattered representations in total for which cancellation takes place within their Dirac cohomology when passing to Dirac index. We will put stars on the KGB elements of these representations. See Section \ref{sec-appendix}. Up to now, whenever cancellation happens, the Dirac index \emph{vanishes}. This is the motivation of Conjecture 1.4 of \cite{D21}.

\begin{example}\label{exam-EVI-cancellation}
Consider the  representation $\pi$ with the following final parameter
\begin{verbatim}
(x=2655,lambda=[2,-1,-1,4,-3,1,3]/1,nu=[4,-1,-4,6,-6,0,5]/2)
\end{verbatim}
It has infinitesimal character $[1,0,0,1,0,1,1]$, which is conjugate to $\rho_c$ under the action of $W(\frg, \frt_f)$. The representation $\pi$ has two spin LKTs:
$$
[1, 0, 1, 0, 0, 4, 2]=\rho_n^{(47)}, \quad  [2, 0, 0, 1, 0, 3, 3]=\rho_n^{(41)}.
$$
Therefore, $H_D(\pi)$ consists of two copies of the trivial $\widetilde{K}$-type.
Note that
$$
w^{(47)}=s_1s_3s_4s_5s_6s_7s_2s_4s_5s_6s_3s_4s_5s_2, \quad
w^{(41)}=s_1s_3s_4s_5s_6s_7s_2s_4s_5s_6s_3s_4s_2.
$$
In particular, $w^{(47)}$ has length $14$ and $w^{(41)}$ has length $13$. Thus one trivial $\widetilde{K}$-type lives in the even part of $H_D(\pi)$,
while the other lives in the odd part of $H_D(\pi)$. See Lemma 2.3 of \cite{DW21}. As a consequence, the Dirac index of $\pi$ vanishes.

On the other hand, it is much quicker to calculate ${\rm DI}(\pi)$ by \texttt{atlas} using \cite{MPVZ} (certain outputs are omitted):
\begin{verbatim}
G:E7_q
set p=parameter(KGB(G,2655),[2,-1,-1,4,-3,1,3],[4,-1,-4,6,-6,0,5]/2)
show_dirac_index(p)
Dirac index is 0
\end{verbatim}
which agrees with the earlier calculation.\hfill\qed
\end{example}

\begin{example}\label{EVI-minimal}
The first entry of Table \ref{table-EVI-1110111} is the minimal representation. Its $K$-types are
$$
\{\mu_n:=\mu_0 + n\beta\mid n\in\bbZ_{\geq 0} \},
$$
where $\mu_0=[0,0,0,0,0,0,4]$ is the unique LKT of $\pi$. We compute that its spin LKTs are $\mu_1, \mu_2, \dots, \mu_6$.
Moreover, both $\mu_1$ and $\mu_6$ contribute a single $\widetilde{K}$-type $[0, 0, 0, 0, 0, 0, 10]$, but with opposite signs. Both $\mu_2$ and $\mu_5$ contribute the following multiplicity-free $\widetilde{K}$-types:
$$
[0, 1, 0, 0, 0, 0, 8], \quad [1, 0, 0, 0, 1, 0, 7], \quad [0, 0, 0, 0, 2, 0, 6].
$$
However, their signs differ on each $\widetilde{K}$-type. Similarly, both $\mu_3$ and $\mu_4$ contribute the following multiplicity-free $\widetilde{K}$-types:
\begin{align*}
&[2, 0, 0, 0, 0, 1, 5], \quad [1, 0, 0, 0, 1, 1, 4], \quad [2, 1, 0, 0, 0, 0, 4], \quad [1, 1, 0, 0, 1, 0, 3],\\
&[0, 0, 0, 1, 0, 1, 3], \quad [0, 1, 0, 1, 0, 0, 2], \quad [0, 0, 2, 0, 0, 0, 0], \quad [2, 0, 0, 1, 0, 0, 0],\\
&[0, 1, 0, 0, 0, 2, 0], \quad [3, 0, 0, 0, 1, 0, 1], \quad [0, 0, 0, 0, 0, 3, 1], \quad [4, 0, 0, 0, 0, 0, 2].
\end{align*}
However, their signs differ on each $\widetilde{K}$-type.
Therefore, the Dirac index of the minimal representation vanishes.

Theorem 6.2 of \cite{DW21} is also helpful here. Indeed, we note that
$$
B(\mu_n-\mu_0, \zeta_7)=B(n\beta, \zeta_7)=n, \quad \forall n\in\bbZ_{\geq 0}.
$$
In particular, it says that $\mu_1/\mu_6$, $\mu_2/\mu_5$, $\mu_3/\mu_4$ have different parities.  \hfill\qed
\end{example}

\section{Special unipotent representations}\label{sec-sup}

In the list \cite{LSU} offered by Adams, the group \texttt{E7\_q} has $56$ \emph{special unipotent representations} (in the sense of \cite{BV}) in total. Ten of them are also fully supported Dirac series representations. We mark them with the subscript $\clubsuit$ in Section \ref{sec-appendix}. Note that the representation in Table \ref{table-EVI-1111111} is the trivial one, and the first entry of Table \ref{table-EVI-1110111} is the minimal representation.

The other eight representations are all $A_{\frq}(\lambda)$ modules. A summary is given below.
Here the first column gives the position of the representation. In the second column, firstly we present the Lie algebra of the connected $L$, the Levi factor of a proper $\theta$-stable parabolic subalgebra $\frq$ of $\frg$. Then we indicate how to get the unitary character $\bbC_{\lambda}$ from the trivial representation \texttt{tL} of $L$. For instance, $a\downarrow 9$ means we should minus the first coordinate of \texttt{lambda(tL)} by $9$. Finally, we give the range of the inducing module $\bbC_{\lambda}$.

\begin{center}
\begin{tabular}{c|c}
special unipotent representation & realization as an $A_q(\lambda)$ module \\
\hline
1st entry of Table \ref{table-EVI-1001011} &  so(8,4).u(1), $a\downarrow 9$, None\\
\hline
2nd entry of Table \ref{table-EVI-1001011} & so(8,4).u(1), $a\downarrow 8$, Fair\\
\hline
3rd entry of Table \ref{table-EVI-1001011} & so*(12)[0,1].u(1), $a\downarrow 9$, None\\
\hline
4th entry of Table \ref{table-EVI-1001011} & so*(12)[0,1].u(1), $a\downarrow 8$, Fair\\
\hline
6th entry of Table \ref{table-EVI-1001011} & so(12).u(1), $a\downarrow 9$, None\\
\hline
17th entry of Table \ref{table-EVI-1001011} & so(12).u(1), $a\downarrow 8$, Fair\\
\hline
1st entry of Table \ref{table-EVI-1110101} & so*(12)[0,1].u(1), $g\downarrow 8$, Fair\\
\hline
2nd entry of Table \ref{table-EVI-1110101} & so(12).u(1), $g\downarrow 8$, Fair\\
\hline
\end{tabular}
\end{center}

\begin{example}\label{EVI-su-Aq}
Let us explain the second row of the above table.
\begin{verbatim}
G:E7_q
set P=KGP(G,[1,2,3,4,5,6])
set L=Levi(P[2])
L
Value: connected real group with Lie algebra 'so(8,4).u(1)'
set tL=trivial(L)
set tm8=parameter(x(tL),lambda(tL)-[8,0,0,0,0,0,0],nu(tL))
goodness(tm8,G)
Value: "Fair"
theta_induce_irreducible(tm8,G)
Value:
1*parameter(x=8916,lambda=[3,0,0,1,0,3,1]/1,nu=[2,0,-1,1,0,2,0]/1) [27]
\end{verbatim}
Other rows are interpreted similarly. \hfill\qed
\end{example}

\section{Improving the Helgason-Johnson bound for $E_{7(-5)}$}\label{sec-fHJ}

As noted in Section \ref{sec-Dirac-series},
the sharpened Helgason-Johnson bound \eqref{new-HJ-bound} has improved the computation efficiency greatly, thus making the classification of $\widehat{E_{7(-5)}}^d$ finishes in about eight months. Now let us further sharpen the Helgason-Johnson bound, hoping that the result may help people to explore (part of) the unitary dual of $E_{7(-5)}$, and that the method will help other cases.

\begin{prop}\label{prop-HJ}
Let $G$ be $E_{7(-5)}$, and let $\pi$ be an irreducible unitary $(\frg, K)$ module with infinitesimal character $\Lambda$ which is given by \eqref{inf-char}. Assume that $\Lambda$ is a non-negative integer combination of $\zeta_1, \dots, \zeta_7$. If
\begin{equation}\label{further-HJ-bound}
\|\nu\| \geq \sqrt{\frac{165}{2}},
\end{equation}
then $\pi$ is either trivial or minimal.
\end{prop}

Note that $E_{7(-5)}$ has a unique minimal representation, namely, the first entry of Table \ref{table-EVI-1110111}. It has GK dim $17$, which equals to $B(\rho, \gamma_7^\vee)$. Recall that $\gamma_7$
defined in \eqref{gamma-7} is the highest root of $\Delta^+(\frg, \frt_f)$, and $\gamma_7^\vee$ is its coroot.

\begin{proof}
Take a LKT $\mu$ of $\pi$. Then its infinitesimal character has the form \eqref{inf-char}. By the Dirac inequality \eqref{Dirac-inequality}, one has that
$$
\|\Lambda\|^2=\|\lambda_a(\mu)\|^2 +\|\nu\|^2\leq \|\mu\|^2_{\rm spin}.
$$
Therefore,
\begin{equation}\label{nu-bound}
\|\nu\|^2\leq \|\mu\|^2_{\rm spin}- \|\mu\|^2_{\rm lambda}.
\end{equation}
As computed in Section 5.8 of \cite{D20}, we have that $\max\{A_j\mid 0\leq j\leq 62\}=82$. We refer the reader to Section 3 of \cite{D20} for the precise meaning of $A_j$. It follows that for any u-large $K$-type $\mu$, one has that
$$\|\mu\|_{\rm spin}^2 - \|\mu\|_{\rm lambda}^2\leq 82.
$$
Since $\|\nu\|\geq\sqrt{82.5}$, we conclude from \eqref{nu-bound} that $\mu$ can \emph{not} be u-large.

There are $53709$ u-small $K$-types in total. Among them, only $23$ have the property that
$$
82.5\leq \|\mu\|_{\rm spin}^2 - \|\mu\|_{\rm lambda}^2.
$$
Let us collect these $23$ u-small $K$-types as \texttt{Certs}. Its elements are listed as follows:
\begin{align*}
&[0, 0, 0, 0, 0, 0, 2], [0, 0, 0, 0, 0, 0, 4], [0, 0, 0, 0, 0, 0, 6], [0, 0, 0, 0, 0, 1, 3],\\
&[0, 0, 0, 0, 0, 1, 5], [0, 0, 0, 0, 0, 1, 7], [0, 0, 0, 0, 0, 2, 4], [0, 0, 0, 0, 0, 2, 6], \\
&[0, 0, 0, 0, 0, 3, 7], [0, 0, 0, 0, 1, 0, 2], [0, 0, 0, 0, 1, 0, 4], [0, 0, 0, 0, 1, 0, 6], \\
&[0, 0, 0, 0, 1, 1, 3], [0, 0, 0, 0, 1, 1, 5], [0, 0, 0, 1, 0, 0, 6], [0, 0, 1, 0, 0, 0, 5], \\
&[0, 1, 0, 0, 0, 0, 4], [0, 1, 0, 0, 0, 0, 6], [1, 0, 0, 0, 0, 0, 3], [1, 0, 0, 0, 0, 0, 5], \\
&[1, 0, 0, 0, 0, 0, 7], [1, 0, 0, 0, 0, 1, 4], [1, 0, 0, 0, 0, 1, 6].
\end{align*}
Moreover, we compute that
$1.5\leq \|\lambda_a(\mu)\|^2\leq 26.5$ for any $\mu\in \texttt{Certs}$. Therefore,
\begin{equation}\label{can-Lambda}
1.5+ 82.5 \leq \|\Lambda\|^2=\|\lambda_a(\mu)\|^2+\|\nu\|^2\leq 26.5 +113.5.
\end{equation}
The right hand side above uses the proven bound \eqref{new-HJ-bound}. There are $1113$ integral $\Lambda$s meeting the requirement \eqref{can-Lambda}. We collect them as $\Omega$.

Now a direct search using \texttt{atlas} says that there are $1754$ irreducible  representations $\pi$ such that $\Lambda\in \Omega$ and that $\pi$ has a LKT which is a member of \texttt{Certs}. Furthermore, only one of them turns out to be unitary: the minimal representation. This finishes the proof.
\end{proof}

\begin{rmk}
We believe that the condition $\Lambda$ being integral is removable. But the proof should be much harder.
\end{rmk}

\begin{example}\label{example-EVI-Phi8-ctd}
Let us continue with Example \ref{example-EVI-Phi8}.
\begin{verbatim}
set fHJ=165/2
set furtherHJ=##for p in allFS do if (nu(p)*TgFWts)*(nu(p)*TgFWts)
<=4*fHJ then [p] else [] fi od
#furtherHJ
Value: 0
\end{verbatim}
Therefore, we now conclude directly that $\Pi_{\rm FS}^{\rm u}([1, 1, 1, 1, 1, 8, 0])=\emptyset$.
\hfill\qed
\end{example}

\begin{example}\label{exam-nus-FS-Dirac-series}
The statistic $\|\nu\|^2$ among the FS-scattered Dirac series is distributed as follows:
\begin{align*}
&7, 8.5, 10, 10.5, 12, 12.5, 13, 13.5^{\underline{2}}, 14.5, 17^{\underline{7}}, 17.5^{\underline{4}}, 18^{\underline{3}}, 25^{\underline{6}},  25.5, 26^{\underline{4}},  \\
&28^{\underline{3}}, 29^{\underline{15}}, 29.5^{\underline{4}}, 30^{\underline{8}},  34^{\underline{4}}, 42^{\underline{2}}, 52, 53.5^{\underline{7}},
54^{\underline{8}}, 73^{\underline{6}}, 73.5, 78^{\underline{7}}, 113.5, 198.
\end{align*}
Here $a^{\underline{n}}$ means the value $a$ occurs $n$ times, and $\underline{n}$ is omitted when $n=1$. Note that $198$ comes from the trivial representation (see Table \ref{table-EVI-1111111}),
while the value $113.5$ comes from the minimal representation (the first entry of Table \ref{table-EVI-1110111}).
\end{example}

\section{Appendix}\label{sec-appendix}

This appendix presents all the $103$  FS-scattered Dirac series representations of $E_{7(-5)}$ according to their infinitesimal characters. Those marked with $\clubsuit$ are special unipotent representations (see Section \ref{sec-sup}). Whenever the even part and the odd part of the Dirac cohomology share the same $\widetilde{K}$-type(s), we will mark the representation with a star (see Section \ref{sec-DI}).

\begin{table}[H]
\centering
\caption{Infinitesimal character $[0,1,1,0,1,0,1]$}
\begin{tabular}{lccc}
 $\# x$ & $\lambda$ & $\nu$ & Spin LKTs   \\
\hline
$8774$ & $[-1,-1,2,1,3,-2,3]$ & $[-2,1,2,-1,3,-2,2]$ & $[0,4,0,0,0,0,6]$, $[0,2,0,2,0,0,2]$,\\
      &                       &                      &$[0,2,0,0,0,4,2]$\\
$8701$ & $[-2,1,1,1,2,-1,2]$ & $[-\frac{3}{2},0,\frac{5}{2},-\frac{1}{2},\frac{5}{2},-\frac{5}{2},\frac{5}{2}]$& $[0,0,0,1,0,0,12]$, $[0,0,0,0,0,6,2]$,\\
       &                     & &$[0,2,0,1,0,0,12]$, $[0,0,0,3,0,0,8]$\\
$6881$ & $[-1,3,2,0,1,-3,5]$ & $[-1,2,2,-1,1,-4,5]$
&$[0,2,0,2,0,0,2]$, $[0,0,4,0,0,0,2]$\\
$6337$ & $[-1,3,2,0,1,-3,5]$ & $[-\frac{3}{2},2,\frac{3}{2},0,0,-\frac{7}{2},5]$
& $[0,0,1,0,3,0,7]$, $[2,0,1,0,1,2,7]$\\
$6103$ & $[1,0,1,-1,3,-1,2]$ & $[-3,1,4,-3,4,-3,1]$
& $[0,1,1,0,1,0,9]$, $[1,1,0,1,1,0,11]$,\\
& & & $[0,0,2,0,0,3,5]$, $[0,2,0,0,2,1,9]$\\
$5296$ & $[-2,1,4,-1,2,-3,6]$ & $[-2,0,3,-1,1,-3,5]$
&$[1,0,0,0,3,0,13]$, $[4,0,0,0,0,3,7]$\\
$5289$ & $[-2,1,4,-3,6,-3,2]$ & $[-2,0,3,-3,5,-3,1]$
&$[4,1,0,0,0,2,2]$, $[4,0,1,0,1,1,0]$,\\
& & & $[4,0,0,2,0,0,2]$\\
$3311$ & $[0,1,3,-2,3,-2,3]$ & $[0,0,3,-3,3,-3,3]$
& $[0,0,0,0,0,0,18]$\\

$2401$ & $[2,2,-3,2,-1,2,0]$ & $[-\frac{5}{2},1,\frac{7}{2},-\frac{5}{2},1,-1,2]$
& $[1,1,1,0,0,3,3]$, $[2,0,0,1,0,3,5]$\\
$1672$ & $[-1,3,4,-3,1,0,3]$ & $[-\frac{3}{2},2,\frac{5}{2},-\frac{5}{2},0,0,2]$ &
$[1,2,0,0,1,1,10]$, $[2,2,0,1,0,0,8]$
\end{tabular}
\label{table-EVI-0110101}
\end{table}

\begin{table}[H]
\centering
\caption{Infinitesimal character $[0,1,1,0,1,1,1]$}
\begin{tabular}{lccc}
 $\# x$ & $\lambda$ & $\nu$ & Spin LKTs   \\
\hline
$8776$ & $[-3,3,4,-1,1,2,1]$ & $[-3,4,4,-3,1,1,1]$  & $[0,0,0,0,2,3,3]$,$[0,0,0,1,2,2,2],$\\
& & & $[0,0,0,0,2,4,4]$,$[0,0,0,1,2,3,3]$\\
$8362$ & $[-2,2,4,-1,1,1,2]$ & $[-3,\frac{5}{2},\frac{11}{2},-\frac{5}{2},0,0,\frac{5}{2}]$  & $[2,0,0,0,0,1,11]$, $[2,1,0,0,0,0,12]$,\\
& & &$[2,0,0,0,0,2,12]$, $[2,1,0,0,0,1,13]$\\
$2903$ & $[-1,1,4,-2,1,2,1]$ & $[-\frac{7}{2},1,\frac{9}{2},-3,0,\frac{3}{2},1]$ & $[0,4,0,0,0,2,8]$, $[0,3,0,1,0,2,6]$,\\
& & &  $[1,2,0,0,1,3,8]$
\end{tabular}
\label{table-EVI-0110111}
\end{table}

\begin{table}[H]
\centering
\caption{Infinitesimal character $[1,0,0,1,0,1,1]$}
\begin{tabular}{lccc}
 $\# x$ & $\lambda$ & $\nu$ & Spin LKTs   \\
\hline
$8916_{\clubsuit}$ & $[2,0,0,1,0,3,1]$ & $[2,0,-1,1,0,2,0]$  & $[0,3,0,0,0,1,9]$, $[0,1,0,0,0,5,1]$,\\
& & &$[0,1,0,2,0,1,5]$\\
$8916_{\clubsuit}$ & $[3,0,0,1,0,3,1]$ & $[2,0,-1,1,0,2,0]$ & $[0,4,0,0,0,0,8]$, $[0,2,0,0,0,4,0]$,\\
& & & $[0,2,0,2,0,0,4]$\\
$8904_{\clubsuit}$ & $[2,0,-1,2,0,3,1]$ & $[\frac{3}{2},0,-\frac{3}{2},\frac{3}{2},0,2,0]$ & $[0,0,0,0,0,1,15]$, $[0,3,0,0,0,1,9]$,\\
& & & $[0,1,0,0,0,5,1]$, $[0,1,0,2,0,1,5]$\\
$8809_{\clubsuit}$ & $[2,-1,-1,3,-1,4,1]$ & $[1,-1,-1,2,-1,3,0]$  & $[0,0,0,1,0,0,14]$, $[0,2,0,1,0,0,10]$,\\
& & & $[0,0,0,0,0,6,0]$, $[0,0,0,3,0,0,6]$\\
$8460^*$ & $[2,1,-3,3,0,0,2]$ & $[1,-3,-3,4,0,1,1]$  & $[0,2,0,0,2,1,7]$, $[0,3,0,0,2,0,6]$\\
& & & $[0,0,2,0,0,3,3]$, $[0,1,2,0,0,2,2]$\\
$6815_{\clubsuit}$ & $[2,-1,-4,6,-1,1,1]$ & $[1,-1,-4,5,-1,0,0]$
& $[0,0,0,0,0,1,15]$\\
$6706^*$ & $[1,0,-4,5,0,1,1]$ & $[0,0,-4,4,-1,1,1]$
& $[4,0,0,0,0,3,5]$, $[4,1,0,0,0,2,4]$\\
$6884^*$ & $[2,-1,-1,3,-1,1,1]$ & $[1,-4,-1,5,-3,1,1]$
& $[3,0,2,0,1,0,3]$, $[2,1,1,0,1,1,4]$\\
$6442$ & $[-1,-1,0,3,-2,3,2]$ & $[1,-2,-2,3,-3,4,1]$
& $[0,0,4,0,0,0,0]$, $[0,2,0,2,0,0,4]$\\
$6598^*$ & $[0,-1,1,1,-1,5,-3]$ & $[\frac{3}{2},-\frac{7}{2},-\frac{3}{2},5,-\frac{7}{2},\frac{3}{2},1]$
& $[4,0,2,0,0,0,4]$,  $[3,1,1,0,0,1,5]$, \\
& & & $[4,0,1,0,1,1,2]$,  $[3,1,0,0,1,2,3]$\\
$6359^*$ & $[3,-3,-2,6,-3,2,1]$ & $[\frac{3}{2},-\frac{7}{2},-\frac{3}{2},5,-\frac{7}{2},2,0]$
& $[0,0,0,0,0,1,15]$, $[0,0,0,1,0,0,14]$ \\
$6357^*$ & $[3,-3,-2,6,-3,2,1]$ & $[\frac{3}{2},-\frac{7}{2},-\frac{3}{2},5,-\frac{7}{2},2,0]$
& $[0,0,0,3,0,0,6]$, $[0,1,0,2,0,1,5]$\\
$5997$ & $[1,0,1,1,-2,3,1]$ & $[0,-\frac{5}{2},-2,4,-3,3,\frac{3}{2}]$
& $[0,0,0,0,4,0,10]$, $[1,0,1,0,2,0,10]$,\\
& & & $[3,0,1,0,0,2,6]$, $[3,0,0,0,1,3,4]$\\
$5764$ & $[1,-2,-1,4,-3,4,2]$ & $[0,-2,-\frac{3}{2},\frac{7}{2},-\frac{7}{2},\frac{7}{2},\frac{3}{2}]$
& $[1,0,0,0,3,0,11]$, $[4,0,0,0,0,3,5]$\\
$5763$ & $[1,-2,-1,4,-3,4,2]$ & $[0,-2,-\frac{3}{2},\frac{7}{2},-\frac{7}{2},\frac{7}{2},\frac{3}{2}]$
 & $[0,0,1,0,3,0,9]$, $[2,0,1,0,1,2,5]$\\
$4964^*$ & $[1,0,1,1,-2,4,-1]$ & $[0,-4,-\frac{1}{2},4,-3,\frac{5}{2},1]$
& $[5,0,1,0,0,1,3]$ , $[4,1,0,0,0,2,4]$\\
$4826_{\clubsuit}$ & $[1,-3,0,4,-3,4,1]$ & $[1,-\frac{7}{2},-\frac{1}{2},\frac{7}{2},-\frac{7}{2},\frac{7}{2},0]$
& $[0,0,0,0,0,0,16]$\\
$4567^*$ & $[1,-1,-3,6,-3,2,1]$ & $[1,-\frac{1}{2},-3,\frac{9}{2},-\frac{7}{2},1,1]$
& $[0,0,2,0,0,0,12]$, $[0,1,1,0,1,0,11]$\\
$4285^*$ & $[1,0,-2,5,-4,3,1]$ & $[1,0,-\frac{5}{2},4,-4,\frac{3}{2},1]$
& $[0,3,0,0,2,0,6]$, $[1,2,0,1,1,0,5]$\\
$4027$ & $[2,0,-1,2,-1,1,2]$ & $[\frac{7}{2},-2,-2,3,-3,\frac{3}{2},1]$
& $[0,4,0,0,0,0,8]$, $[1,2,0,0,1,1,8]$,\\
& & & $[1,1,0,1,1,1,6]$, $[0,2,0,2,0,0,4]$\\
$3742$ & $[4,-1,-2,3,-2,2,2]$ & $[4,-\frac{3}{2},-\frac{5}{2},\frac{5}{2},-\frac{5}{2},\frac{3}{2},1]$
& $[6,0,0,0,2,0,0]$, $[6,0,0,1,0,0,2]$\\
$3741$ & $[4,-1,-2,3,-2,2,2]$ & $[4,-\frac{3}{2},-\frac{5}{2},\frac{5}{2},-\frac{5}{2},\frac{3}{2},1]$
& $[0,0,0,1,0,0,14]$\\
$3740$ & $[4,-1,-2,3,-2,2,2]$ & $[4,-\frac{3}{2},-\frac{5}{2},\frac{5}{2},-\frac{5}{2},\frac{3}{2},1]$
& $[1,3,0,0,1,0,7]$, $[1,2,0,1,1,0,5]$\\
$3739$ & $[4,-1,-2,3,-2,2,2]$ & $[4,-\frac{3}{2},-\frac{5}{2},\frac{5}{2},-\frac{5}{2},\frac{3}{2},1]$
&$[0,3,0,0,0,1,9]$, $[0,1,0,2,0,1,5]$\\
$3187^*$ & $[3,-1,-3,4,-2,2,2]$ & $[\frac{5}{2},-1,-\frac{7}{2},\frac{7}{2},-\frac{5}{2},1,\frac{3}{2}]$
& $[7,0,0,0,1,0,1]$, $[6,0,0,1,0,0,2]$\\
$2655^*$ & $[2,-1,-1,4,-3,1,3]$ & $[2,-\frac{1}{2},-2,3,-3,0,\frac{5}{2}]$
& $[1,0,1,0,0,4,2]$, $[2,0,0,1,0,3,3]$\\
$1521$ & $[2,-1,-2,4,-2,1,2]$ & $[1,-2,-2,3,-\frac{3}{2},0,\frac{3}{2}]$
& $[4,0,0,0,0,3,5]$, $[3,1,1,0,0,1,5]$,\\
& & & $[2,0,1,0,1,2,5]$, $[3,1,0,0,1,2,3]$
\end{tabular}
\label{table-EVI-1001011}
\end{table}

\begin{table}[H]
\centering
\caption{Infinitesimal character $[1,0,0,1,1,0,1]$}
\begin{tabular}{lccc}
 $\# x$ & $\lambda$ & $\nu$ & Spin LKTs   \\
\hline
$5972$  & $[1,-2,-1,4,1,-3,5]$  & $[1,-2,-2,3,1,-4,5]$ & $[0,0,3,1,0,0,1]$, $[0,1,1,2,0,0,3]$\\
$5247$  & $[1,-1,-1,3,1,-2,4]$  & $[0,-2,-\frac{3}{2},\frac{7}{2},0,-\frac{7}{2},5]$ & $[2,0,0,0,3,0,12]$, $[5,0,0,0,0,3,6]$\\
$5246$  & $[1,-1,-1,3,1,-2,4]$  & $[0,-2,-\frac{3}{2},\frac{7}{2},0,-\frac{7}{2},5]$
& $[1,0,1,0,3,0,8]$, $[2,0,1,0,2,1,6]$\\
$4283$  & $[1,-2,0,3,1,-2,3]$  & $[1,-\frac{7}{2},-\frac{1}{2},\frac{7}{2},0,-\frac{7}{2},\frac{7}{2}]$
& $[1,0,0,0,0,0,17]$
\end{tabular}
\label{table-EVI-1001101}
\end{table}

\begin{table}[H]
\centering
\caption{Infinitesimal character $[1,0,0,1,1,1,1]$}
\begin{tabular}{lccc}
 $\# x$ & $\lambda$ & $\nu$ & Spin LKTs   \\
\hline
$8556$  & $[2,1,-1,2,-1,1,2]$  & $[1,-4,-4,5,1,1,1]$ & $[0,0,0,0,3,3,3]$\\
$7888$  & $[2,-1,-3,4,1,1,2]$  & $[\frac{5}{2},-\frac{5}{2},-\frac{11}{2},\frac{11}{2},0,0,\frac{5}{2}]$ & $[3,0,0,0,0,1,12]$
\end{tabular}
\label{table-EVI-1001111}
\end{table}

\begin{table}[H]
\centering
\caption{Infinitesimal character $[1,0,1,1,0,1,1]$}
\begin{tabular}{lccc}
 $\# x$ & $\lambda$ & $\nu$ & Spin LKTs   \\
\hline
$7273$  & $[1,-1,2,2,-1,1,1]$  & $[0,-5,2,5,-4,1,1]$ & $[0,4,0,0,0,0,0]$, $[0,4,0,0,0,2,2]$\\
$7070$  & $[-1,-1,2,2,-1,2,1]$  & $[1,-\frac{9}{2},2,\frac{9}{2},-\frac{9}{2},2,0]$ & $[0,0,0,1,0,2,12]$, $[0,0,0,0,0,6,8]$,\\
& & & $[0,0,0,3,0,0,14]$
\end{tabular}
\label{table-EVI-1011011}
\end{table}

\begin{table}[H]
\centering
\caption{Infinitesimal character $[1,1,0,1,0,1,0]$}
\begin{tabular}{lccc}
 $\# x$ & $\lambda$ & $\nu$ & Spin LKTs   \\
\hline
$8385$  & $[1,-1,-1,3,0,1,0]$  & $[0,1,-2,4,-3,3,-2]$ &  $[0,4,0,0,0,0,4]$, $[0,2,0,2,0,0,0]$,\\
& & & $[0,2,0,0,0,4,4]$\\
$8272$  & $[-1,1,0,3,-1,2,-1]$  & $[1,0,-\frac{5}{2},\frac{9}{2},-\frac{5}{2},\frac{5}{2},-\frac{5}{2}]$ & $[0,0,0,1,0,0,10]$, $[0,0,0,0,0,6,4]$, \\
& & & $[0,2,0,1,0,0,14]$, $[0,0,0,3,0,0,10]$\\
$2450$  & $[2,1,-1,2,-1,2,-1]$  & $[\frac{5}{2},1,-\frac{5}{2},2,-3,3,-2]$ & $[1,1,0,1,1,0,13]$, $[1,0,0,2,1,0,11]$\\
$2419$  & $[1,1,0,2,0,-1,0]$  & $[\frac{5}{2},0,-\frac{5}{2},2,-\frac{5}{2},\frac{7}{2},-\frac{5}{2}]$ & $[6,0,0,0,2,0,4]$, $[6,0,0,1,0,0,6]$
\end{tabular}
\label{table-EVI-1101010}
\end{table}

\begin{table}[H]
\centering
\caption{Infinitesimal character $[1,1,0,1,0,1,1]$}
\begin{tabular}{lccc}
 $\# x$ & $\lambda$ & $\nu$ & Spin LKTs   \\
\hline
$8893$  & $[3,1,0,1,0,1,1]$  & $[3,1,-2,3,-2,1,1]$ &  $[0,0,0,2,1,1,1], [0,0,0,1,1,3,3]$,\\
& & & $0,0,0,0,1,5,5]$\\
$8673$  & $[3,1,0,2,-1,1,2]$  & $[3,0,0,\frac{5}{2},-\frac{5}{2},0,\frac{5}{2}]$ & $[1,0,0,0,0,1,10]$, $[1,1,0,0,0,1,12]$, \\
& & & $[1,2,0,0,0,1,14]$\\
$5479$  & $[2,2,0,-1,1,2,1]$  & $[5,1,-4,1,-2,3,1]$ & $[1,3,0,0,2,0,1]$, $[2,2,1,0,2,0,3]$\\
$5170$  & $[1,1,-1,3,-2,2,1]$ & $[1,1,-4,5,-4,1,1]$ &  $[0,1,2,0,1,0,10]$, $[0,0,3,0,0,2,7]$, \\
        &                     &                     & $[0,2,1,0,2,0,11]$\\
$4893$  & $[4,2,-2,1,-1,3,1]$  & $[5,\frac{3}{2},-\frac{7}{2},0,-\frac{3}{2},\frac{7}{2},0]$ & $[1,0,0,0,0,2,15]$, $[0,1,0,0,1,1,17]$\\
$4891$  & $[4,2,-2,1,-1,3,1]$  & $[5,\frac{3}{2},-\frac{7}{2},0,-\frac{3}{2},\frac{7}{2},0]$ & $[0,0,0,2,1,2,6]$, $[0,0,1,1,0,3,8]$\\
$4287$  & $[1,1,-1,3,-2,2,1]$  & $[1,0,-3,5,-5,2,1]$ &  $[5,0,0,0,0,3,0]$, $[5,0,0,1,0,2,1]$\\
$3693$  & $[3,1,-2,1,0,3,1]$  & $[\frac{7}{2},0,-\frac{7}{2},0,0,3,1]$ &  $[6,1,0,0,1,0,6]$, $[5,2,0,0,0,1,4]$\\
$2321$  & $[2,1,-1,2,-1,1,2]$  & $[3,0,-3,3,-3,0,3]$ & $[0,0,0,0,1,0,20]$
\end{tabular}
\label{table-EVI-1101011}
\end{table}

\begin{table}[H]
\centering
\caption{Infinitesimal character $[1,1,0,1,1,0,1]$}
\begin{tabular}{lccc}
$\# x$ & $\lambda$ & $\nu$ & Spin LKTs   \\
\hline
$4952$  & $[0,2,2,-1,1,1,0]$  & $[5,1,-4,1,1,-3,4]$ &  $[0,4,0,0,2,0,0]$, $[1,3,0,0,3,0,1]$,\\
        &                     &                     & $[1,3,1,0,2,0,2]$\\
$4354$  & $[4,2,-2,1,2,-2,3]$  & $[5,\frac{3}{2},-\frac{7}{2},0,2,-\frac{7}{2},\frac{7}{2}]$ & $[2,0,0,0,0,2,16]$, $[1,0,0,0,1,2,17]$,\\
        &                     &                     & $[1,1,0,0,1,1,18]$\\
$4352$  & $[4,2,-2,1,2,-2,3]$  & $[5,\frac{3}{2},-\frac{7}{2},0,2,-\frac{7}{2},\frac{7}{2}]$ & $[0,0,0,1,2,3,7]$, $[0,0,1,0,1,4,9]$,\\
        &                      &                    &$[0,0,1,1,1,3,8]$\\
$3162$  & $[3,1,-2,1,2,-1,4]$  & $[\frac{7}{2},0,-\frac{7}{2},0,3,-3,4]$ & $[6,2,0,0,0,0,6]$, $[7,1,0,0,1,0,7]$,\\
        &                      &                    &$[6,2,0,0,0,1,5]$\\
\end{tabular}
\label{table-EVI-1101101}
\end{table}

\begin{table}[H]
\centering
\caption{Infinitesimal character $[1,1,0,1,1,1,1]$}
\begin{tabular}{lccc}
 $\# x$ & $\lambda$ & $\nu$ & Spin LKTs   \\
\hline
$6443$  & $[0,1,1,2,-1,1,1]$  & $[7,0,-7,2,1,1,1]$&  $[0,6,0,0,0,0,0]+n\beta$, $0\leq n\leq 1$\\
$6152$  & $[4,1,-1,1,1,1,1]$  & $[\frac{15}{2},0,-\frac{13}{2},2,0,2,0]$ & $[ 0,0,0,0,0,5,13]$, $[0,0,0,0,0,6,12]$\\
& & & $[0,0,0,1,0,4,14]$
\end{tabular}
\label{table-EVI-1101111}
\end{table}

\begin{table}[H]
\centering
\caption{Infinitesimal character $[1,1,1,0,1,0,1]$}
\begin{tabular}{lccc}
$\# x$ & $\lambda$ & $\nu$ & Spin LKTs   \\
\hline
$8935_{\clubsuit}$  & $[1,2,0,-1,3,-1,2]$  & $[1,1,2,-1,2,-1,1]$& $[0,0,0,1,0,3,3]$, $[0,0,0,2,0,3,3]$,\\
& & & $[0,0,0,a,0,6-2a,6-2a]$, $0\leq a\leq 3$\\
$8817_{\clubsuit}$  & $[3,1,-1,1,2,-1,2]$  & $[\frac{5}{2},0,1,-\frac{1}{2},\frac{5}{2},-\frac{5}{2},\frac{5}{2}]$& $[0,a,0,0,0,1,9+2a]$, $0\leq a\leq 4$,\\
& & & $[0,2,0,0,0,0,12]$, $[0,1,0,0,0,2,12]$\\
$7928$  & $[1,4,2,-3,4,-1,1]$  & $[0,5,2,-4,4,-3,1]$& $[0,4,0,0,0,0,2]$,\\
& & & $[0,3,0,0,0,2,2]$, $[0,3,0,0,0,3,3]$\\
& & & $[0,3,0,1,0,1,1]$, $[0,3,0,1,0,2,2]$\\
$7829$  & $[-1,3,2,-2,3,0,1]$  & $[1,\frac{9}{2},2,-\frac{9}{2},\frac{9}{2},-\frac{5}{2},0]$&  $[ 0,0,0,1,0,1,11]$, $[0,1,0,0,0,2,12]$, \\
& & &  $[0,0,0,2,0,0,12]$, $[0,1,0,1,0,1,13]$, \\
& & &  $[0,0,0,0,0,6,6]$, $[0,0,0,3,0,0,12]$\\
$7836^*$  & $[2,2,1,0,1,-2,4]$  & $[1,1,1,0,1,-5,6]$& $[0,0,1,0,3,0,3]$, $[0,1,1,0,3,0,5],$\\
& & & $[0,0,2,0,2,1,3]$, $[0,1,2,0,2,0,6]$,\\
& & & $[0,0,4,0,0,0,6]$, $[0,0,3,0,1,1,4]$\\
$5449$  & $[1,3,1,0,1,-2,3]$  & $[1,4,1,-1,0,-4,4]$&  $[2,0,0,0,0,0,16]$, $[1,0,0,0,1,0,17]$,\\
&  &  & $[0,0,1,0,1,0,19]$\\
$5316$  & $[1,4,1,-1,2,-3,5]$  & $[0,4,0,-1,1,-4,5]$&  $[6,0,0,0,0,2,6]$, $[6,0,0,1,0,0,8]$,\\
& & & $[6,0,0,0,0,3,5]$\\
$5481$  & $[1,1,5,-2,1,-2,5]$  & $[1,1,4,-3,1,-3,4]$& $[0,3,0,0,2,0,0]$, $[1,2,1,0,2,0,2]$,\\
& & & $[1,2,2,0,1,0,3]$\\
$5480$  & $[-1,2,1,0,1,0,2]$  & $[1,1,1,-3,4,-4,5]$& $[0,0,0,3,0,0,0]$, $[0,0,2,2,0,0,0]$,\\
& & &$[0,1,0,3,0,0,2]$\\
$5072$  & $[1,0,2,2,-1,-2,2]$  & $[0,\frac{3}{2},2,-4,4,-3,\frac{9}{2}]$&  $[2,0,0,0,4,0,10]$, $[3,0,1,0,2,0,10]$,\\
& & & $[4,0,0,0,2,2,6]$, $[4,0,1,0,1,1,8]$,\\
& & & $[4,0,2,0,0,0,10]$\\
$4901$  & $[2,2,3,-2,2,-2,3]$  & $[\frac{3}{2},\frac{3}{2},\frac{7}{2},-\frac{7}{2},2,-\frac{7}{2},\frac{7}{2}]$
&$[2,0,0,0,0,1,15]$, $[1,0,0,0,1,1,16]$,\\
& & & $[0,1,0,0,2,0,18]$\\
$4899$  & $[2,2,3,-2,2,-2,3]$  & $[\frac{3}{2},\frac{3}{2},\frac{7}{2},-\frac{7}{2},2,-\frac{7}{2},\frac{7}{2}]$&
$[0,0,0,1,2,2,6]$, $[0,0,1,1,1,2,7]$,\\
& & & $[0,0,2,0,0,3,9]$\\
$4800$  & $[1,2,2,-2,3,-2,4]$  & $[0,\frac{3}{2},2,-\frac{7}{2},\frac{7}{2},-\frac{7}{2},5]$&
$[3,0,0,0,3,0,11]$, $[5,0,0,0,1,2,7]$, \\
& & &$[5,0,1,0,0,1,9]$\\
$4799$  & $[1,2,2,-2,3,-2,4]$  & $[0,\frac{3}{2},2,-\frac{7}{2},\frac{7}{2},-\frac{7}{2},5]$& $[2,0,1,0,3,0,9]$, $[3,0,1,0,2,1,7]$,\\
& & & $[3,0,2,0,1,0,9]$\\
$3750$  & $[1,1,3,-2,3,-2,3]$  & $[1,0,3,-\frac{7}{2},\frac{7}{2},-\frac{7}{2},\frac{7}{2}]$&  $[0,1,0,0,0,0,18]$, $[0,0,0,1,0,0,20]$\\
$3605$  & $[1,1,3,-2,2,-1,4]$  & $[0,0,\frac{7}{2},-\frac{7}{2},3,-3,4]$&  $[7,0,0,0,1,0,7]$, $[6,1,0,0,0,1,5]$,\\
& & & $[6,1,0,0,0,2,4]$\\
$3538$  & $[1,-3,2,0,2,1,0]$  & $[1,\frac{9}{2},1,-\frac{7}{2},1,-\frac{3}{2},\frac{5}{2}]$&  $[0,1,2,0,2,0,6]$, $[0,0,4,0,0,0,6]$,\\
& & & $[0,2,1,1,1,0,7]$\\
$3275$  & $[2,4,2,-3,2,-1,3]$  & $[1,4,\frac{3}{2},-4,\frac{3}{2},-\frac{3}{2},\frac{5}{2}]$& $[3,0,0,0,1,4,1]$, $[3,0,1,0,0,4,2]$\\
$3274$  & $[2,4,2,-3,2,-1,3]$  & $[1,4,\frac{3}{2},-4,\frac{3}{2},-\frac{3}{2},\frac{5}{2}]$&  $[2,0,2,0,0,0,12]$, $[1,0,2,0,1,0,13]$,\\
& & & $[3,0,1,0,0,2,12]$\\
$2467$  & $[1,1,2,-2,4,-2,2]$  & $[1,1,1,-3,4,-3,1]$& $[4,0,0,2,0,0,6]$, $[3,0,0,2,1,0,7]$,\\
        &   &  &   $[4,0,1,1,1,0,5]$, $[4,1,0,1,0,1,7]$\\
$1867$  & $[2,2,1,-1,1,0,2]$  & $[\frac{5}{2},\frac{5}{2},0,-\frac{5}{2},0,0,\frac{5}{2}]$& $[8,0,0,0,0,1,3]$, $[8,0,0,1,0,0,2]$\\
$1760$  & $[2,2,1,-2,3,-1,2]$  & $[2,2,0,-3,3,-2,1]$&  $[0,2,0,0,2,0,14]$, $[0,2,0,0,2,1,13]$, \\
& & & $[1,1,0,1,1,0,15]$
\end{tabular}
\label{table-EVI-1110101}
\end{table}

\begin{table}[H]
\centering
\caption{Infinitesimal character $[1,1,1,0,1,1,1]$}
\begin{tabular}{lccc}
$\# x$ & $\lambda$ & $\nu$ & Spin LKTs   \\
\hline
$8940^*_{\clubsuit}$  & $[2,1,1,0,1,1,1]$  & $[\frac{3}{2},\frac{3}{2},\frac{3}{2},-\frac{3}{2},\frac{3}{2},2,0]$  & $[0, 0, 0, 0, 0, 0, 4]$ + $n\beta$, $1\leq n\leq 6$\\
$8398$  & $[1,1,1,-2,5,1,1]$  & $[1,1,1,-5,6,1,1]$  & $[0, 0, 0, 0, 4, 0, 0]$ + $n\beta$, $0\leq n\leq 3$\\
$7611$  & $[2,2,1,-2,3,1,2]$  & $[\frac{5}{2},3,0,-\frac{11}{2},\frac{11}{2},0,\frac{5}{2}]$  & $[4,0,0,0,0,0,12]+n\beta$, $0\leq n\leq 3$\\
$6875$  & $[1,1,3,-1,1,1,1]$  & $[0,0,7,-5,1,1,1]$  & $[0,5,0,0,0,0,0]+n\beta$, $0\leq n\leq 2$\\
$6640$  & $[1,1,3,-1,1,1,1]$  & $[1,0,\frac{13}{2},-\frac{9}{2},0,2,0]$  & $[0,0,0,0,0,4,12]$, $[0,0,0,1,0,3,13]$,\\
& & & $[0,0,0,0,0,6,10]$, $[0,0,0,2,0,2,14]$\\
$4653$  & $[1,4,1,-2,2,1,1]$  & $[1,6,1,-5,1,1,1]$  & $[0,0,4,0,0,0,10]$, $[0,0,4,0,0,1,9]$,\\
&  &  & $[0,1,3,0,1,0,11]$\\
$3806$  & $[1,3,2,-2,1,2,1]$  & $[1,5,2,-5,0,2,1]$  & $[6,0,0,0,0,4,0]$, $[6,0,0,1,0,3,1]$\\
$1869$  & $[2,2,1,-1,1,1,2]$  & $[3,3,0,-3,0,0,3]$  &  $[0,0,0,0,2,0,22]$, $[0,0,1,0,1,0,23]$
\end{tabular}
\label{table-EVI-1110111}
\end{table}

\begin{table}[H]
\centering
\caption{Infinitesimal character $[1,1,1,1,0,1,0]$}
\begin{tabular}{lccc}
$\# x$ & $\lambda$ & $\nu$ & Spin LKTs   \\
\hline
$3985$ & $[1,2,1,1,-2,4,-1]$ & $[0,\frac{3}{2},2,0,-4,\frac{11}{2},-\frac{9}{2}]$ &  $[4,0,0,0,4,0,10]$, $[5,0,0,0,3,1,8]$, \\
& & & $[5,0,1,0,2,0,10]$
\end{tabular}
\label{table-EVI-1111010}
\end{table}

\begin{table}[H]
\centering
\caption{Infinitesimal character $[1,1,1,1,0,1,1]$}
\begin{tabular}{lccc}
$\# x$ & $\lambda$ & $\nu$ & Spin LKTs   \\
\hline
$8192$ & $[1,1,2,-1,2,0,1]$ & $[1,1,1,1,-6,7,1]$ & $[0,0,0,0,5,0,0]+n\beta$, $0\leq n\leq 2$\\
$7267$ & $[2,2,1,1,-2,3,2]$ & $[\frac{5}{2},3,0,0,-\frac{11}{2},\frac{11}{2},\frac{5}{2}]$ & $[5,0,0,0,0,0,13]+n\beta$, $0\leq n\leq 2$
\end{tabular}
\label{table-EVI-1111011}
\end{table}

\begin{table}[H]
\centering
\caption{Infinitesimal character $[1,1,1,1,1,0,1]$}
\begin{tabular}{lccc}
$\# x$ & $\lambda$ & $\nu$ & Spin LKTs   \\
\hline
$7939$ & $[2,1,1,1,1,-1,3]$ & $[1,1,1,1,1,-7,8]$ & $[0,0,0,0,6,0,0]+n\beta$, $0\leq n\leq 1$ \\
$6854$ & $[2,1,1,1,1,-1,3]$ & $[\frac{5}{2},3,0,0,0,-\frac{11}{2},8]$ & $[6,0,0,0,0,0,14]+n\beta$, $0\leq n\leq 1$
\end{tabular}
\label{table-EVI-1111101}
\end{table}

\begin{table}[H]
\centering
\caption{Infinitesimal character $[1,1,1,1,1,1,1]$}
\begin{tabular}{lccc}
$\# x$ & $\lambda$ & $\nu$ & Spin LKT   \\
\hline
$8945_{\clubsuit}$ & $[1,1,1,1,1,1,1]$ & $[1,0,1,2,0,2,0]$ & $[0,0,0,0,0,0,0]$
\end{tabular}
\label{table-EVI-1111111}
\end{table}

\centerline{\scshape Funding}
Dong is supported by the National Natural Science Foundation of China (grant 12171344).

\medskip

\centerline{\scshape Acknowledgements}
We are deeply grateful to the \texttt{atlas} mathematicians. In particular, we thank van Leeuwen for improving the \texttt{atlas} script \texttt{print\_branch\_irr\_long} from the efficiency of $N$ to $\log(N)$. For instance, a previous calculation which gave no result after one month now finishes within 24 hours. Using the updated script, branching of the $\#8398$ representation in Table \ref{table-EVI-1110111} to the \texttt{atlas} height $303$ (as required by Theorem \ref{thm-HP}) took about twenty days. Thus it should be a mission impossible without the updated branching script. We also thank Kei Yuen Chan and Kayue Daniel Wong for helpful suggestion/discussion on improving the Helgason-Johnson bound. Finally, we sincerely thank the two referees for giving us excellent suggestions.


\begin{thebibliography}{99}
\bibitem{ALTV} J.~Adams, M.~van Leeuwen, P.~Trapa, and D.~Vogan, \emph{Unitary representations of real reductive groups}, Ast{\'erisque} \textbf{417} (2020).

\bibitem{BDW} D.~Barbasch, C.-P.~Dong, K.D.~Wong, \emph{Dirac series for complex classical Lie groups: A multiplicity one theorem}, Adv. Math. \textbf{403} (2022), Paper Number 108370, 47pp.

\bibitem{BP19} D.~Barbasch, P.~Pand\v zi\'c,  \emph{Twisted Dirac index and applications
to characters}, Affine, vertex and $W$-algebras, pp.~23--36, Springer INdAM Series \textbf{37}, 2019.

\bibitem{BV}
D.~Barbasch, D.~Vogan,
\emph{Unipotent representations of complex semisimple Lie groups},
Ann. of Math. \textbf{121} (1985), 41--110.

\bibitem{BW} A.~Borel and N.~Wallach,
\emph{Continuous cohomology, discrete subgroups, and representations
of reductive groups}, 2nd ed., Mathematical Surveys and Monographs, vol.~\textbf{67}, Amer.
Math. Soc., Providence, RI, 2000.



\bibitem{Ca} J.~Carmona, \emph{Sur la classification des modules admissibles irr\'eductibles}, pp.11--34 in Noncommutative Harmonic Analysis and Lie Groups, J. Carmona and M. Vergne, eds., Lecture
Notes in Mathematics \textbf{1020}, Springer-Verlag, New York, 1983.


\bibitem{DDH}  L.-G.~Ding,  C.-P.~Dong and H.~He,
\emph{Dirac series for $E_{6(-14)}$}, J. Algebra \textbf{590} (2022), 168--201.

\bibitem{DDY}  J.~Ding,   C.-P.~Dong and L.~Yang,
\emph{Dirac series for some real exceptional Lie groups}, J. Algebra \textbf{559} (2020), 379--407.

\bibitem{D13}  C.-P.~Dong,
\emph{On the Dirac cohomology of complex Lie group representations},
Transform. Groups \textbf{18} (1) (2013), 61--79. [Erratum: Transform.
Groups \textbf{18} (2) (2013), 595--597.]

\bibitem{D17} C.-P.~Dong,
\emph{Unitary representations with Dirac cohomology: finiteness in the real case},
Int. Math. Res. Not. IMRN \textbf{2020} (24), 10277--10316.

\bibitem{D20} C.-P.~Dong,
\emph{On the Helgason-Johnson bound},
Israel J. Math., accepted. See also arXiv:2110.00694.

\bibitem{D21} C.-P.~Dong,
\emph{A non-vanishing criterion for Dirac cohomology}, Transformation groups, https://doi.org/10.1007/s00031-022-09758-0.


\bibitem{DW20p} C.-P.~Dong and K.D.~Wong,
\emph{Dirac series of $GL(n, \bbR)$}, Int. Math. Res. Not. IMRN, https://doi.org/10.1093/imrn/rnac150.

\bibitem{DW21} C.-P.~Dong and K.D.~Wong,
\emph{Dirac index of some unitary representations of $Sp(2n, \bbR)$ and $SO^*(2n)$}, J. Algebra \textbf{603} (2022), 1--37.




\bibitem{HJ} S.~Helgason, K.~Johnson, \emph{The bounded spherical functions on symmetric spaces}, Adv. Math.
\textbf{3} (1969), 586--593.

\bibitem{HP} J.-S. Huang and P.~Pand\v zi\'c, \emph{Dirac
cohomology, unitary representations and a proof of a conjecture of
Vogan}, J. Amer. Math. Soc.  \textbf{15} (2002), 185--202.



\bibitem{Kn} A.~Knapp, \emph{Lie Groups, Beyond an Introduction}, Birkh\"{a}user, 2nd Edition, 2002.

\bibitem{Ko} B.~Kostant, \emph{Lie algebra cohomology and the generalized Borel-Weil theorem},
Ann. of Math. \textbf{74} (1961), 329--387.

\bibitem{KV} A.~Knapp and D.~Vogan,  \emph{Cohomological induction and unitary
representations}, Princeton Univ. Press, Princeton, N.J., 1995.


\bibitem{MPVZ} S.~Mehdi, P.~Pand\v zi\'c,  D.~Vogan and R.~Zierau, \emph{Dirac index and associated cycles of Harish-Chandra modules}, Adv. Math. \textbf{361} (2020), 106917, 34 pp.


\bibitem{Pa} R.~Parthasarathy, \emph{Dirac operators and the discrete
series}, Ann. of Math. \textbf{96} (1972), 1--30.

\bibitem{Pa2} R.~Parthasarathy, \emph{Criteria for the unitarizability of some highest weight modules},
Proc. Indian Acad. Sci. \textbf{89} (1) (1980), 1--24.

\bibitem{Paul} A.~Paul, \emph{Cohomological induction in Atlas}, slides of July 14, 2017,  available from {\rm http://www.liegroups.org/workshop2017/workshop/presentations/Paul2HO.pdf.}

\bibitem{PRV} R.~Parthasarathy, R.~Ranga Rao, and
S.~Varadarajan, \emph{Representations of complex semi-simple Lie
groups and Lie algebras}, Ann. of Math. \textbf{85} (1967),
383--429.

\bibitem{Sa} S.~Salamanca-Riba, \emph{On the unitary dual of real reductive Lie groups and the $A_{\mathfrak{q}}(\lambda)$ modules: the
strongly regular case}, Duke Math. J. \textbf{96} (3) (1999),  521--546.

\bibitem{SV} S.~Salamanca-Riba, D.~Vogan, \emph{On the classification of unitary representations of reductive Lie
groups}, Ann. of Math. \textbf{148} (3) (1998), 1067--1133.

\bibitem {Vog80} D.~Vogan,
\emph{Singular unitary representations},  Noncommutative harmonic
analysis and Lie groups (Marseille, 1980),  506--535.

\bibitem {Vog81} D.~Vogan, \emph{Representations of real reductive Lie groups}, Birkh\"auser, 1981.


\bibitem{Vog84} D.~Vogan, \emph{Unitarizability of certain series of representations},
Ann. of Math. \textbf{120} (1984), 141--187.


\bibitem{Vog97} D.~Vogan, \emph{Dirac operators and unitary
representations}, 3 talks at MIT Lie groups seminar, Fall 1997.

\bibitem{Vog22} D.~Vogan, \emph{Classifying the unitary dual (part 1 of infinitely many...)}, \texttt{atlas} seminar, June 30, 2022.

\bibitem{At} Atlas of Lie Groups and Representations, version 1.1, March 2022. See www.liegroups.org for more about the software.

\bibitem{LSU} A complete list of special unipotent representations of real exceptional groups, see  http://www.liegroups.org/tables/unipotentExceptional/.
\end{thebibliography}
\end{document}